\documentclass[11pt]{article}%
\usepackage{amssymb,amsmath,amsfonts,amsthm,array,bm,bbm,color}%
\usepackage{amsmath}%
\usepackage{amsfonts}%
\usepackage{amssymb}%
\usepackage{graphicx}
\providecommand{\U}[1]{\protect\rule{.1in}{.1in}}
\numberwithin{equation}{section}

\newtheorem{theorem}{Theorem}[section]
\newtheorem{lemma}[theorem]{Lemma}
\newtheorem{corollary}[theorem]{Corollary}
\newtheorem{proposition}[theorem]{Proposition}
\newtheorem{remark}[theorem]{Remark}

\newtheorem{definition}[theorem]{Definition}
\newtheorem{conjecture}[theorem]{Conjecture}
\newtheorem{problem}[theorem]{Problem}
\def\<{\langle}
\def\>{\rangle}
\def\d{{\rm d}}
\def\L{\mathcal{L}}
\def\div{{\rm div}}
\def\E{\mathbb{E}}

\def\N{\mathbb{N}}
\def\P{\mathbb{P}}
\def\R{\mathbb{R}}
\def\T{\mathbb{T}}
\def\Z{\mathbb{Z}}

\def\eps{\varepsilon}

\begin{document}

\title{Scaling limit of stochastic 2D Euler equations with transport noises to the deterministic Navier--Stokes equations}
\author{Franco Flandoli\footnote{Email: franco.flandoli@sns.it. Scuola Normale Superiore of Pisa, Piazza dei Cavalieri 7, 56124 Pisa, Italy.}
\quad Lucio Galeati\footnote{Email: lucio.galeati@iam.uni-bonn.de. Institute of Applied Mathematics, University of Bonn, Germany.}
\quad Dejun Luo\footnote{Email: luodj@amss.ac.cn. Key Laboratory of RCSDS, Academy of Mathematics and Systems Science, Chinese Academy of Sciences, Beijing 100190, China, and School of Mathematical Sciences, University of the Chinese Academy of Sciences, Beijing 100049, China. } }

\maketitle

\vspace{-10pt}

\begin{abstract}
We consider a family of stochastic 2D Euler equations in vorticity form on the torus, with transport type noises and $L^2$-initial data. Under a suitable scaling of the noises, we show that the solutions converge weakly to that of the deterministic 2D Navier--Stokes equations. Consequently, we deduce that the weak solutions of the stochastic 2D Euler equations are approximately unique and ``weakly quenched exponential mixing''.
\end{abstract}

\section{Introduction}

Let $\T^2=\R^2/\Z^2$ be the 2D torus and $\Z^2_0= \Z^2\setminus \{0\}$ the nonzero lattice points. Define
  $$\sigma_k(x)= \frac{k^\perp}{|k|} e_k(x),\quad x\in \T^2,\ k\in \Z^2_0,$$
where $k^\perp= (k_2, -k_1)$ and $\{e_k \}_{k\in \Z^2_0}$ is the usual trigonometrical basis of $L^2(\T^2)$, see the beginning of Section 2. Then $\{\sigma_k\}_{k\in \Z^2_0}$ is a complete orthonormal basis of the space of square integrable, divergence free vector fields on $\T^2$ with zero mean.

In a previous work \cite{FlaLuo-2}, the first and the third named authors studied the vorticity form of the stochastic 2D Euler equations with transport type noise:
  \begin{equation}\label{approx-eq}
  \d \xi^N + u^N\cdot\nabla\xi^N\,\d t= 2\sqrt{\nu}\,\eps_N \sum_{|k|\leq N} \frac1{|k|} \sigma_{k} \cdot \nabla\xi^N \circ \d W^{k},
  \end{equation}
where $\nu>0$ is a constant, $\{\eps_N\}_{N\geq 1}$ a sequence of positive numbers and $\{W^k\}_{k\in \Z^2_0}$ is a family of independent standard Brownian motions on some filtered probability space $(\Omega, \mathcal F, \mathcal F_t, \P)$. In the above equation, $u^N= (u^N_1, u^N_2)$ is the velocity field and $\xi^N= \nabla^\perp \cdot u^N= \partial_2 u^N_1 - \partial_1 u^N_2$ is the vorticity; conversely, $u^N =K\ast \xi^N$ where $K$ is the Biot--Savart kernel. The equation \eqref{approx-eq} has the enstrophy measure $\mu$ on $\T^2$ as the invariant measure, which is supported on $H^{-1-}(\T^2) = \bigcap_{s<-1} H^s(\T^2)$, $H^s(\T^2)$ being the usual Sobolev space on $\T^2$. For any fixed $N\geq 1$, it is known that \eqref{approx-eq} admits a stationary solution $\xi^N$ with paths in $C\big([0,T], H^{-1-}(\T^2)\big)$ (taking $\rho_0\equiv 1$ in \cite[Theorem 1.3]{FlaLuo-1}). We choose the parameter $\eps_N$ in such a way that it compensates the coefficient appearing in the It\^o--Stratonovich correction term. More precisely, let
  $$\eps_N= \bigg(\sum_{|k|\leq N} \frac1{|k|^2} \bigg)^{-1/2} \sim (\log N)^{-1/2},$$
then, in the It\^o formulation, \eqref{approx-eq} becomes
  $$\d \xi^N + u^N\cdot\nabla\xi^N\,\d t= 2\sqrt{\nu}\,\eps_N \sum_{|k|\leq N} \frac1{|k|} \sigma_{k} \cdot \nabla\xi^N \, \d W^{k} + \nu\Delta \xi^N\,\d t.$$
It was proved in \cite{FlaLuo-2} that the stationary solutions $\xi^N$ of \eqref{approx-eq} converge to the unique-in-law stationary solution of the stochastic 2D Navier--Stokes equations driven by additive space-time white noise
  \begin{equation}\label{NS-vort}
  \d\xi +u\cdot\nabla\xi\,\d t=\nu\Delta\xi \,\d t+ \sqrt{2\nu } \, \nabla^{\perp} \cdot\d W.
  \end{equation}
Here, $W= \sum_{k\in \Z^2_0} \sigma_{k} W^{k}$ is a cylindrical Brownian motion in the Hilbert space of divergence free vector fields on $\mathbb T^2$. The equation \eqref{NS-vort}, in the velocity form, has been studied intensively in the past three decades, see for instance \cite{AC, AF, DaPD}. In particular, it was shown in \cite{DaPD} that \eqref{NS-vort} has a pathwise unique strong solution for $\mu$ almost every initial data in Besov spaces of negative order. As a consequence of the Yamada--Watanabe type theorem (see e.g. \cite{Kur}), the stationary solutions of \eqref{NS-vort} are unique in law.

On the other hand, the second named author considered in \cite{Gal} a similar scaling limit for a sequence of stochastic transport linear equations, but in a different regime, namely for function-valued solutions of suitable regularity. To state the result we introduce the notation $\ell^p\ (p\in [1,\infty])$ which are the usual spaces of real sequences indexed by $\Z^2_0$ and denote the norm by $\|\cdot\|_{\ell^p}$. Let $\theta^N_\cdot \in \ell^2\, (N\geq 1)$ be a sequence verifying, for all $N\in \N$,
  \begin{equation}\label{theta}
  \theta^N_k = \theta^N_j \quad \mbox{whenever } |k|= |j|
  \end{equation}
and
  \begin{equation}\label{theta-N}
  \lim_{N\to\infty} \frac{\|\theta^N_\cdot\|_{\ell^\infty}} {\|\theta^N_\cdot\|_{\ell^2}} =0.
  \end{equation}
The main result of \cite{Gal} asserts that, if
  \begin{equation}\label{eps-N}
  \eps_N = \frac{2\sqrt{\nu}} {\|\theta^N_\cdot\|_{\ell^2}},
  \end{equation}
then the solutions $\xi^N$ of the sequence of stochastic transport linear equations ($b$ is a vector field on $\T^2$)
  $$\d \xi^N = b\cdot \nabla \xi^N\,\d t + \eps_N \sum_{k\in \Z^2_0} \theta^N_k \sigma_k \cdot \nabla \xi^N \circ \d W^k$$
converge to the solution of the parabolic Cauchy problem
  $$\partial_t \xi= \nu \Delta \xi+  b\cdot \nabla \xi, \quad \xi(0)= \xi_0.$$
For $\xi_0\in L^2(\T^2)$, this last equation admits a unique weak solution in $L^2\big(0,T; L^2(\T^2) \big)$ under mild assumptions on $b$, see for instance \cite[Lemma 3.3]{Gal}.

Motivated by the above discussions, we consider, in the regime of regular solutions (compared to the white noise solutions considered in \cite{FlaLuo-2}), the stochastic 2D Euler equations
  \begin{equation}\label{SEE-vort}
  \d \xi^N + u^N\cdot\nabla\xi^N\,\d t= \eps_N \sum_{k\in \Z^2_0} \theta^N_k \sigma_{k} \cdot \nabla\xi^N \circ \d W^{k},
  \end{equation}
where $\big\{ \theta^N_\cdot \big\}_{N\geq 1}$ satisfies \eqref{theta} and \eqref{theta-N}, and $\eps_N$ is defined as in \eqref{eps-N}. We assume $\xi^N_0 = \xi_0\in L^2(\T^2)$ with zero mean. Then one can show that the equation \eqref{SEE-vort} admits a solution $\xi^N$ (weak in both analytic and probabilistic sense), satisfying
  $$\sup_{t\in [0,T]} \int_{\T^2} |\xi^N(t,x)|^2 \,\d x <+\infty.$$
We will prove that such more regular solutions of equation \eqref{SEE-vort} converge to the unique solution of the \textit{deterministic} 2D Navier--Stokes equations
  \begin{equation}\label{NSE}
  \partial_{t}\xi +u\cdot\nabla\xi=\nu\Delta\xi, \quad \xi(0)=\xi_0 .
  \end{equation}
According to the classical theory of 2D Navier--Stokes equations (see \cite[Theorem 3.2]{Teman} for the velocity form), the above equation has a unique solution.

A direct consequence of the above scaling limit is that the transport type noises considered here regularize the 2D Euler equations asymptotically. More precisely, it is well known that the 2D Euler equations has a unique solution if the initial data $\xi_0$ belongs to $L^\infty(\T^2)$, while the uniqueness of solutions remains an open problem in the case $\xi_0\in L^p(\T^2)$ for $p<\infty$. Although we cannot prove that the stochastic 2D Euler equation \eqref{SEE-vort} has a unique solution for $L^2(\T^2)$-initial data, the above result shows that, in the limit, we get the uniquely solvable 2D Navier--Stokes equation \eqref{NSE}. As a result, the distances between the laws of weak solutions of \eqref{SEE-vort} tend to zero as $N\to \infty$. We call such a property the approximate weak uniqueness, see Section \ref{subsec-uniqueness} for more details.

Our main result is the convergence to deterministic Navier--Stokes equations; however, tuning parameters in the right way we may construct sequences converging to deterministic Euler equation. More precisely, given any viscosity solution of 2D Euler, we can find a suitable sequence converging to it, see Section \ref{sec 6.2} for more details. We do not know the converse, namely if every limiting measure constructed in this way is a superposition of viscosity solutions, but our result makes this conjecture plausible. It is very important to identify selection criteria, for instance by viscosity, by noise or by additional physical requirements, in view of the multiplicity of solutions found recently by the method of convex integration \cite{DeLellSze};\ although our result is not conclusive, it makes plausible that the zero-noise limit selects viscosity solutions. Notice that this is different from what happens for certain examples of linear transport equations \cite{AttFla}.

Our result also has interesting implications related to the mixing behavior of incompressible flows, a phenomenon which recently attracted a lot of attention, see for instance \cite{ACM14, ACM19, YZ} and the references therein. In \cite{ACM19}, Alberti et al. considered the solutions to the continuity equation
  \begin{equation}\label{CE}
  \partial_t \rho_t + \div(\rho_t u) =0
  \end{equation}
and estimated the ``mixedness'' of $\rho_t$ as $t\to \infty$ in terms of the negative Sobolev norm $\|\rho_t\|_{\dot H^{-1}}$. Here $\dot H^{s}(\T^2)\, (s\in\R)$ denote the homogeneous Sobolev spaces. They constructed a bounded and divergence free vector field $u\in C^\infty \big([0,\infty) \times \R^d, \R^d\big)$ and a bounded solution $\rho\in C^\infty \big([0,\infty) \times \R^d \big)$ to \eqref{CE} such that, for any $0<s<2$, it holds
  $$\|\rho_t\|_{\dot H^{-s}} \leq C_s\, e^{-c s t}, \quad t\geq 0,$$
where $C_s>0$ and $c>0$ are constants. Such exponential mixing result is in fact optimal, taking into account the lower bounds on functional mixing scale proved in \cite{IKX, Seis}. Using our limit result and the exponential decay of the energy and the enstrophy of the solution to the Navier--Stokes equation \eqref{NSE}, we prove in Section \ref{subsec-mixing} that the solutions to the stochastic 2D Euler equations \eqref{SEE-vort} satisfy the ``weakly quenched exponential mixing'' property, a notion inspired by discussions in \cite[Section 7]{CoKiRyZl}.

However, the decay in $\dot{H}^{-s}$-norms does not extend to the $L^{2}$-norm and our result does not imply anomalous dissipation of enstrophy. This is a difficult open question, which is discussed in Section \ref{subsec-anomalous}.

This paper is organized as follows. In Section 2, we introduce some notations and state the main results, including the existence of weak solutions to the stochastic 2D Euler equations \eqref{SEE-vort} and the scaling limit to the deterministic 2D Navier--Stokes equation \eqref{NSE}, as well as a finite dimensional convergence result. The proofs of these results are provided in Sections 3 to 5. In the last sections, we discuss the consequences of the scaling limit in more detail.

\section{Functional settings and main results}\label{sec2}

In this section, we give some more notations for functional spaces and state the main results of the paper. Let $C^\infty(\T^2)$ be the space of smooth function on $\T^2$. We write $\<\cdot , \cdot \>$ and $\|\cdot \|_{L^2}$ for the inner product and the norm in $L^2(\T^2)$. Recall also the Sobolev spaces $H^s(\T^2),\, s\in \R$. Denote by
  $$e_k(x) = \sqrt{2} \begin{cases}
  \cos(2\pi k\cdot x), & k\in \Z^2_+, \\
  \sin(2\pi k\cdot x), & k\in \Z^2_-,
  \end{cases} \quad x\in \T^2, $$
where $\Z^2_+= \{k\in \Z^2_0: (k_1>0) \mbox{ or } (k_1=0,\, k_2>0)\}$ and $\Z^2_-= -\Z^2_+$. To save notations, we shall write the vector valued spaces $L^2(\T^2, \R^2)$ and $H^s(\T^2, \R^2)$ simply as $L^2(\T^2)$ and $H^s(\T^2)$. We denote by $H$ (resp. $V$) the subspace of $L^2(\T^2)$ (resp. $H^1(\T^2)$) of functions with zero mean. Moreover, we assume $\{W^k\}_{k\in \Z^2_0}$ is a family of independent $\mathcal F_t$-Brownian motions on the filtered probability space $(\Omega, \mathcal F, \mathcal F_t, \P)$.

First, we fix $\theta_\cdot \in \ell^2$ verifying \eqref{theta} and consider the following stochastic 2D Euler equation in vorticity form:
  \begin{equation}\label{SEE-vort-1}
  \d \xi + u\cdot \nabla\xi\,\d t=  \eps \sum_{k\in \Z^2_0} \theta_k \sigma_{k} \cdot \nabla\xi \circ \d W^{k}, \quad \xi(0)= \xi_0\in H,
  \end{equation}
with
  \begin{equation}\label{epsilon}
  \eps =\frac{2\sqrt{\nu}}{\|\theta_\cdot\|_{\ell^2}} .
  \end{equation}
Using \eqref{theta}, it is not difficult to prove the simple equality (cf. \cite[Lemma 2.6]{FlaLuo-2})
  \begin{equation}\label{useful-equality}
  \sum_{k\in \Z^2_0} \theta_k^2\, \sigma_k(x)\otimes \sigma_k(x) \equiv \frac12 \|\theta_\cdot\|^2_{\ell^2} I_2,\quad x\in\mathbb{T}^2,
  \end{equation}
where $I_2$ is the $2\times 2$ identity matrix. From this we deduce the It\^o formulation of \eqref{SEE-vort-1}:
  \begin{equation}\label{SEE-vort-2}
  \d \xi + u\cdot \nabla\xi\,\d t= \nu \Delta\xi\,\d t + \eps \sum_{k\in \Z^2_0} \theta_k \sigma_{k} \cdot \nabla\xi\, \d W^{k}.
  \end{equation}
This equation is understood as follows: for any $\phi \in C^\infty(\T^2)$, it holds $\P$-a.s. for all $t\in [0,T]$,
  \begin{equation}\label{SEE-formulation}
  \<\xi_t,\phi\> = \<\xi_0,\phi\> + \int_0^t \<\xi_s, u_s\cdot \nabla\phi\>\,\d s + \nu \int_0^t \<\xi_s, \Delta\phi\>\,\d s - \eps \sum_{k\in \Z^2_0} \theta_k \int_0^t \<\xi_s, \sigma_k\cdot \nabla\phi\>\,\d W^k_s.
  \end{equation}
Recall that $u$ is related to $\xi$ via the Biot--Savart kernel $K$ on $\T^2$: $u=K\ast \xi$. Thus if $\xi\in L^2 \big(\Omega,L^2 (0,T; H) \big)$, then $u\in L^2 \big(\Omega,L^2 (0,T; V) \big)$. Under this condition, if $\xi$ (and also $u$) is $\mathcal F_t$-progressively measurable, it is clear that all the terms in the above equation makes sense. For instance, the stochastic integral is a square integrable martingale since
  $$\aligned
  \E \Bigg(\sum_{k\in \Z^2_0} \theta_k^2 \int_0^t \< \xi_s, \sigma_k\cdot \nabla\phi\>^2\,\d s \Bigg) &
  \leq \|\theta_\cdot\|_{\ell^\infty}^2 \E \Bigg( \int_0^T \sum_{k\in \Z^2_0}\< \xi_s\nabla\phi, \sigma_k \>^2\,\d s\Bigg)\\
  &\leq \|\theta_\cdot\|_{\ell^\infty}^2 \E \bigg( \int_0^T \Vert\xi_s\nabla\phi \Vert_{L^2}^2\,\d s\bigg)\\
  &\leq \|\theta_\cdot\|_{\ell^\infty}^2 \|\nabla\phi\|_\infty^2 \E \int_0^T \|\xi_s\|_{L^2}^2 \,\d s <+\infty,
  \endaligned$$
  where we used the fact that $\{\sigma_k\}_{k\in\mathbb{Z}_0^2}$ form an (incomplete) orthonormal system in $L^2(\T^2,\R^2)$.
From this result we can give the definition of solutions to \eqref{SEE-formulation}.

\begin{definition}\label{SEE-def}
We say that \eqref{SEE-formulation} has a weak solution if there exist a filtered probability space $\big( \Omega, \mathcal F, \mathcal F_t, \P\big)$, a sequence of independent $\mathcal F_t$-Brownian motions $\{W^k\}_{k\in \Z^2_0}$ and an $\mathcal F_t$-progressively measurable process $\xi\in L^2 \big(\Omega,L^2 (0,T; H) \big)$ with $\mathbb{P}$-a.s. weakly continuous trajectories such that for any $\phi \in C^\infty(\T^2) $, the equality \eqref{SEE-formulation} holds $\P$-a.s. for all $t\in [0,T]$.
\end{definition}

Note that the solution is weak in both the probabilistic and the analytic sense. Our first result is the existence of solutions to \eqref{SEE-formulation}.

\begin{theorem}\label{thm-existence}
For any $\xi_0\in H$, there exists at least one weak solution to \eqref{SEE-formulation} with trajectories in $L^\infty(0,T; H)$; more precisely,
\begin{equation}\label{energy estimate thm-existence}
\sup_{t\in [0,T]} \|\xi_t \|_{L^2} \leq \|\xi_0 \|_{L^2}\quad \mathbb{P}\mbox{-a.s.}
\end{equation}
\end{theorem}

Next, we take a sequence $\theta^N \in \ell^2$, satisfying \eqref{theta} and \eqref{theta-N}, and consider the stochastic 2D Euler equations \eqref{SEE-vort}. Similarly to the above discussions, \eqref{SEE-vort} is understood as follows: for any $\phi\in C^\infty(\T^2)$ and $t\in [0,T]$,
  \begin{equation}\label{SEE-approx-formulation}
  \big\<\xi^N_t,\phi \big\> = \<\xi_0,\phi\> + \int_0^t \big\<\xi^N_s, u^N_s\cdot \nabla\phi \big\>\,\d s + \nu \int_0^t \big\<\xi^N_s, \Delta\phi \big\>\,\d s - \eps_N \sum_{k\in \Z^2_0} \theta^N_k \int_0^t \big\<\xi^N_s, \sigma_k\cdot \nabla\phi \big\>\,\d W^k_s.
  \end{equation}
We remark that Theorem \ref{thm-existence} only provides us with weak solutions, thus the processes $\xi^N_\cdot$ might be defined on different probability spaces. The relevant notion of convergence of these processes is the weak convergence of their laws. Here is the main result of this paper.

\begin{theorem}\label{thm-main}
Assume that the conditions \eqref{theta}--\eqref{eps-N} hold. Let $Q^N$ be the law of $\xi^N_\cdot$, $N\geq 1$. Then the family $\big\{ Q^N \big\}_{N\geq 1}$ is tight in $C([0,T];H^-)$ and it converges weakly to $\delta_{\xi_\cdot}$, where $\xi_\cdot$ is the unique solution of the 2D Navier--Stokes equations \eqref{NSE}.
\end{theorem}

Theorem \ref{thm-main} also implies convergence of the associated advected passive scalars, see Corollary \ref{corollary sec4} for the precise statement.

\begin{remark}\label{rem-theta}
If $\xi_0 \in L^\infty(\T^2)$, then under slightly stronger conditions on $\theta_\cdot$ (e.g. assume $\theta_k \sim |k|^{-2-\delta}$ for some $\delta>0$), the equation \eqref{SEE-vort-1} has a unique solution in $L^\infty \big( [0,T] \times \T^2 \big)$, see for instance \cite[Theorem 2.10]{BrFlMa}. Note that in the approximating equations \eqref{SEE-approx-formulation}, we can take $\theta^N\in \ell^2$ such that there are only finitely many $k$ for which $\theta^N_k\neq 0$, and at the same time satisfying \eqref{theta-N}, for instance, $\theta^N_k = {\bf 1}_{\{|k|\leq N \}}$.  Therefore, if we approximate $\xi_0\in L^2(\T^2)$ by a sequence of bounded functions $\xi^N_0 \in L^\infty(\T^2)$, then the approximating sequence $\xi^N_\cdot\, (N\geq 1)$ are unique solutions of the equations \eqref{SEE-approx-formulation}. Moreover in this case we can consider the sequence $\xi^N_\cdot$ to be defined on the same probability space $(\Omega, \mathcal F, \P)$, again by the results in \cite{BrFlMa}; thus convergence in law to a deterministic limit implies also convergence in probability. The energy bound \eqref{energy estimate thm-existence} then also implies convergence in $L^p(\Omega, \mathbb{P})$, for any $p< \infty$.
\end{remark}

In Sections \ref{sec3} and \ref{sec4}, we prove Theorems \ref{thm-existence} and \ref{thm-main} respectively. Then in Section \ref{sec5} we show that the same result can be achieved, under the same scaling, already working with finite dimensional approximations of Galerkin type. Theorem \ref{thm sec5} below shows that, in some sense, the two limit procedures contained in Theorems \ref{thm-existence} and \ref{thm-main} can be united into a single approximation. Moreover, it provides explicit examples, for any $\xi_0\in H$, of H\"older continuous in time, spatially smooth, enstrophy-preserving functions which converge in $C([0,T];H^-)$ to the dissipative solution of 2D Navier--Stokes equation with initial data $\xi_0$. See Section \ref{subsec-anomalous} for further discussions. Denote by $\Pi_N$ the orthogonal projection of $L^2(\T^2)$ into $H_N=\text{span}\{e_k: k\in\mathbb{Z}^2_0, \vert k\vert\leq N \}$; we consider for each $N$ the solution $\tilde\xi^N$ of the SDE:
  \begin{equation}\label{SDE sec2}
  \d \tilde\xi^N = -\Pi_N\big( \big(K\ast \tilde\xi^N \big)\cdot\nabla\tilde\xi^N \big)\,\d t + \varepsilon_N \sum_{k\in\mathbb{Z}^2_0}\theta^N_k \Pi_N\big( \sigma_k\cdot\nabla \tilde\xi^N \big)\circ\d W^k, \quad \tilde\xi^N_0=\Pi_N \xi_0.
  \end{equation}
The variables $\big\{\tilde\xi^N \big\}_{N\in \N}$ are defined on the same probability space with respect to the same Brownian motions $\{W^k \}_{k\in\mathbb{Z}^2_0}$. In this case we can prove the following

\begin{theorem}\label{thm sec5}
Suppose the sequence $\big\{\theta^N \big\}_{N\geq 1}\subset \ell^2$ satisfies \eqref{theta}--\eqref{eps-N}, and the additional condition
\begin{equation}\label{condition thm5}
\lim_{N\to\infty} \big\Vert \theta^N_\cdot \big\Vert_{\ell^2}^{-2} \sum_{k: \vert k-j\vert>N} \big(\theta^N_k \big)^2 =0\quad \forall \, j\in\mathbb{Z}_0^2.
\end{equation}
Then the sequence $\big\{ \tilde\xi^N \big\}_{N\geq 1}$ converges in $C([0,T];H^-)$ to $\xi_\cdot$ in probability, where $\xi_\cdot$ is the unique solution of Navier--Stokes equation \eqref{NSE} with initial data $\xi_0$.
\end{theorem}

\begin{remark}
\begin{itemize}
\item[\rm (i)] It can be checked for instance that condition \eqref{condition thm5} is satisfied for $\big\{\theta^N_\cdot \big\}_N$ given by $\theta^N_k = \vert k\vert^{-\alpha}\,{\bf 1}_{\{\vert k\vert\leq N \}}$, for any $\alpha\in [0,1]$.
\item[\rm (ii)] The dissipative nature of the limit equation \eqref{NSE} implies that convergence in $C([0,T]; H)$ cannot hold and, therefore, that higher Sobolev norms of such sequences must explode; see also Remark \ref{rem end sec5}.
\end{itemize}
\end{remark}

\section{Existence of solutions to \eqref{SEE-formulation}}\label{sec3}

In this section we give a proof of Theorem \ref{thm-existence} by using the Galerkin approximation and the compactness method.

To use the method of Galerkin approximation, we introduce some notations. For $N\geq 1$, let $H_N={\rm span}\{e_k: k\in \Z^2_0, |k|\leq N\}$ which is a finite dimensional subspace of $H$. Denote by $\Pi_N: H\to H_N$ the orthogonal projection: $\Pi_N \xi= \sum_{|k|\leq N} \<\xi, e_k\> e_k$. $\Pi_N$ can also act on vector valued functions. Let
  $$b_N(\xi) = \Pi_N\big( (K\ast \Pi_N\xi) \cdot \nabla(\Pi_N\xi) \big), \quad  G_N^k(\xi)= \Pi_N\big( \sigma_k \cdot \nabla(\Pi_N\xi) \big),\quad k\in \Z^2_0.$$
Note that, for fixed $N$, there are only finitely many $k\in \Z^2_0$ such that $G_N^k$ is not zero. We shall view $b_N$ and $G_N^k$ as vector fields on $H_N$ whose generic element is denoted by $\xi_N$. These vector fields have the following useful properties:
  \begin{equation}\label{properties}
  \big\< b_N(\xi_N), \xi_N \big\> = \big\< G_N^k(\xi_N), \xi_N \big\>=0\quad \mbox{for all } \xi_N\in H_N,
  \end{equation}
which can be proved easily from the definitions of $b_N$ and $G_N^k$, and the integration by parts formula.
Consider the finite dimensional version of \eqref{SEE-vort-2} on $H_N$:
  \begin{equation}\label{SDE}
  \d\xi_N(t)= -b_N(\xi_N(t))\,\d t + \nu \Delta \xi_N(t)\,\d t + \eps \sum_{k\in \Z^2_0} \theta_k G_N^k(\xi_N(t)) \, \d W^{k}_t, \quad \xi_N(0)= \Pi_N\xi_0,
  \end{equation}
where $\xi_0\in H$ is the initial condition in Theorem \ref{thm-existence}. We remark that the sum over $k$ is a finite sum. Its generator is
  $$\L_N \varphi(\xi_N) = \< - b_N(\xi_N) +\nu \Delta \xi_N, \nabla_N \varphi(\xi_N)\> + \frac{\eps^2}2 \sum_{k\in \Z^2_0} \theta_k^2 \, {\rm Tr}\big[ \big(G_N^k \otimes G_N^k \big) \nabla_N^2 \varphi \big](\xi_N)$$
for any $\varphi\in C_b^2(H_N)$, where $\nabla_N$ is the gradient operator on $H_N$.

\begin{lemma}\label{energy-estimate}
The equation \eqref{SDE} has a unique strong solution $\xi_N(t)$ satisfying
  $$\sup_{t\in [0,T]} \|\xi_N(t) \|_{L^2} \leq \|\xi_N(0) \|_{L^2} \quad \P \mbox{\rm-}a.s.$$
\end{lemma}

\begin{proof}
The vector fields $b_N$ and $G^k_N$ are respectively quadratic and linear on the finite dimensional space $H_N$, therefore they are smooth. By the standard SDE theory, local existence and uniqueness of strong solutions to \eqref{SDE} holds for any initial data. By the It\^o formula,
  \begin{equation}\label{energy-estimate-1}
  \aligned
  \d\|\xi_N(t) \|_{L^2}^2=&\ -2 \<\xi_N(t), b_N(\xi_N(t))\>\,\d t +2\nu \<\xi_N(t), \Delta \xi_N(t) \> \,\d t\\
  &\ + 2\,\eps \sum_{k\in \Z^2_0} \theta_k \big\<\xi_N(t), G_N^k(\xi_N(t)) \big\>\, \d W^{k}_t+ \eps^2 \sum_{k\in \Z^2_0} \theta_k^2 \big\| G_N^k(\xi_N(t))\big\|_{L^2}^2\,\d t.
  \endaligned
  \end{equation}
The first and the third terms on the right hand side vanish due to \eqref{properties}. Moreover, noting that $\Pi_N:H \to H_N$ is an orthogonal projection,
  \begin{equation}\label{energy-estimate-2}
  \big\| G_N^k(\xi_N(t))\big\|_{L^2} = \big\| \Pi_N( \sigma_k \cdot \nabla \xi_N(t))\big\|_{L^2} \leq \| \sigma_k \cdot \nabla \xi_N(t)\|_{L^2}.
  \end{equation}
Therefore,
  $$\eps^2 \sum_{k\in \Z^2_0} \theta_k^2 \big\| G_N^k(\xi_N(t))\big\|_{L^2}^2 \leq \eps^2 \sum_{k\in \Z^2_0} \theta_k^2 \int_{\T^2} (\sigma_k \cdot \nabla \xi_N(t))^2 \,\d x = 2\nu \| \nabla \xi_N(t) \|_{L^2}^2, $$
where the last equality is due to \eqref{useful-equality} and \eqref{epsilon}. Combining these results with \eqref{energy-estimate-1} we obtain $\d \|\xi_N(t) \|_{L^2}^2 \leq 0$, which implies the desired inequality and also the global existence of solution to \eqref{SDE}.
\end{proof}

\begin{remark}
Due to the inequality \eqref{energy-estimate-2}, the solution of \eqref{SDE} does not preserve the $L^2$-norm, unlike the equations \eqref{SDE sec2}. Therefore, although \eqref{SEE-vort-2} is formally equivalent to the conservative Stratonovich equation \eqref{SEE-vort-1}, as the Galerkin approximation of \eqref{SEE-vort-2}, the equation \eqref{SDE} is no longer conservative.
\end{remark}

Lemma \ref{energy-estimate} shows that $\{\xi_N(\cdot)\}_{N\geq 1}$ is bounded in $L^p \big(\Omega,L^p(0,T; H) \big)$ for any $p>2$:
  \begin{equation}\label{key-bound}
  \E \int_0^T \|\xi_N(t) \|_{L^2}^p \,\d t \leq T\|\xi_N(0) \|_{L^2}^p \leq T\|\xi_0 \|_{L^2}^p.
  \end{equation}
Thus we can find a weakly convergent subsequence. Denote by $u_N = K\ast \xi_N,\, N\geq 1$; then $\{u_N(\cdot)\}_{N\geq 1}$ is bounded in $L^2 \big(\Omega, L^2(0,T; V) \big)$. In order to pass to the limit in the nonlinear term, we need $u_N$ to be strongly convergent in $L^2 \big(\Omega,L^2(0,T; H) \big)$. In fact, we will show that the laws $\eta_N$ of $u_N(\cdot)$ are tight in $C\big([0,T], H^{1-}(\T^2) \big)$. To this end we first recall the compactness result by J. Simon \cite[Corollary 9, p.90]{Simon}.

Take any $\delta\in (0,1)$ small enough and $\beta> 4$ (this choice is due to computations below the proof of Lemma \ref{lem-estimate}). We have the compact inclusions
  $$V= H^1 \subset H^{1-\delta} \subset H^{-\beta},$$
and there exists $C>0$ such that
  $$\|f\|_{H^{1-\delta}} \leq C \|f\|_V^{1-\kappa} \|f\|_{H^{-\beta}}^{\kappa}, \quad f\in V,$$
where $\kappa = \delta/(1+\beta)$. Recall that, for $\alpha\in (0,1)$, $p>1$ and a normed linear space $(Y, \|\cdot \|_Y)$, the fractional Sobolev space $W^{\alpha, p}(0,T; Y)$ is defined as those functions $f\in L^p(0,T; Y)$ such that
  $$\int_0^T\! \int_0^T \frac{\|f(t)-f(s) \|_Y^p}{|t-s|^{1+ \alpha p}}\,\d t\d s< +\infty. $$
The next result follows from \cite[Corollary 9, p.90]{Simon}.

\begin{theorem} \label{thm-simon}
Let $\beta>4$ be given. If $p> 12 (1+\beta -\delta)/\delta$, then
  $$L^p(0,T; V) \cap W^{1/3, 4} \big(0,T; H^{-\beta} \big) \subset C\big( [0,T]; H^{1-\delta} \big)$$
with compact inclusion.
\end{theorem}

If we can prove that $\{\eta_N \}_{N\in \N}$ are tight on $C\big( [0,T]; H^{1-\delta} \big)$ for any $\delta\in (0,1)$, then the tightness of $\{\eta_N \}_{N\in \N}$ on $C\big([0,T], H^{1-}(\T^2) \big)$ follows immediately.

To show the tightness of $\{\eta_N\}_{N\geq 1}$ on $ C\big( [0,T]; H^{1-\delta} \big)$, by Theorem \ref{thm-simon}, it is sufficient to prove, for each $N\geq 1$,
  \begin{equation}\label{tightness-estimate-1}
  \E \int_0^T \|u_N(t)\|_V^p\, \d t + \E \int_0^T\! \int_0^T \frac{\|u_N(t) - u_N(s)\|_{H^{-\beta}}^4}{|t-s|^{7/3}}\,\d t\d s \leq C.
  \end{equation}
By \eqref{key-bound}, we immediately get the uniform boundedness of $\{u_N(\cdot)\}_{N\geq 1}$ in $L^p \big(\Omega, L^p(0,T; V) \big)$. It remains to estimate the second expected value.

\begin{lemma}\label{lem-estimate}
There is a constant $C>0$ such that for any $N\geq 1$ and $0\leq s<t\leq T$,
  $$\E \big(\<\xi_N(t) - \xi_N(s), e_k\>^4 \big) \leq C |k|^8 |t-s|^2 \quad \mbox{for all } k\in \Z^2_0.$$
\end{lemma}

\begin{proof}
It is enough to consider $|k|\leq N$. By \eqref{SDE}, we have
  \begin{equation}\label{lem-estimate-1}
  \aligned
  \<\xi_N(t) - \xi_N(s), e_k\> =&\, \int_s^t \<\xi_N(r), u_N(r)\cdot \nabla e_k\>\,\d r + \nu \int_s^t \<\xi_N(r), \Delta e_k\>\,\d r \\
  & - \eps \sum_{l\in \Z^2_0} \theta_l \int_s^t \<\xi_N(r), \sigma_l\cdot \nabla e_k\>\,\d W^l_r.
  \endaligned
  \end{equation}
Using the H\"older inequality and Lemma \ref{energy-estimate}, we obtain
  $$\aligned
  \E \bigg( \Big| \int_s^t \<\xi_N(r), u_N(r)\cdot \nabla e_k\>\,\d r \Big|^4 \bigg) &\leq |t-s|^3\, \E \int_s^t \<\xi_N(r), u_N(r)\cdot \nabla e_k\>^4\,\d r \\
  & \leq |t-s|^3\, \E \int_s^t \|\xi_N(r)\|_{L^2}^4 \| u_N(r) \|_{L^2}^4 \|\nabla e_k\|_\infty^4 \,\d r \\
  &\leq C \|\xi_0\|_{L^2}^8 |k|^4 |t-s|^4,
  \endaligned$$
where the last step is due to the fact $\nabla e_k = 2\pi k e_{-k}$. In the same way, since $\Delta e_k=- 4\pi^2\vert k\vert^2\,e_k$,
  $$\E \bigg( \Big| \int_s^t \<\xi_N(r), \Delta e_k\>\,\d r \Big|^4 \bigg) \leq C \|\xi_0\|_{L^2}^4 |k|^8 |t-s|^4.$$
Next, by Burkholder's inequality,
  $$\aligned
  \E \bigg( \Big| \eps \sum_{l\in \Z^2_0} \theta_l \int_s^t \<\xi_N(r), \sigma_l\cdot \nabla e_k\>\,\d W^l_r \Big|^4 \bigg) &\leq C \eps^4\, \E \bigg( \Big| \sum_{l\in \Z^2_0} \theta_l^2 \int_s^t \<\xi_N(r), \sigma_l\cdot \nabla e_k\>^2\,\d r \Big|^2 \bigg) .
  \endaligned $$
We have
  $$\aligned
  \sum_{l\in \Z^2_0} \theta_l^2 \<\xi_N(r), \sigma_l\cdot \nabla e_k\>^2 &\leq \|\theta \|_{\ell^\infty}^2 \sum_{l\in \Z^2_0} \<\xi_N(r)\nabla e_k, \sigma_l\>^2 \\
  &\leq \|\theta \|_{\ell^\infty}^2 \|\xi_N(r)\nabla e_k \|_{L^2}^2 \leq C\|\theta \|_{\ell^\infty}^2 |k|^2 \|\xi_0\|_{L^2}^2,
  \endaligned$$
where we have used the fact that $\{\sigma_l\}_{l\in \Z^2_0}$ is an orthonormal family. Therefore,
  $$\E \bigg( \Big| \eps \sum_{l\in \Z^2_0} \theta_l \int_s^t \<\xi_N(r), \sigma_l\cdot \nabla e_k\>\,\d W^l_r \Big|^4 \bigg) \leq C\eps^4 \|\theta \|_{\ell^\infty}^4 |k|^4 \|\xi_0\|_{L^2}^4 |t-s|^2 \leq C' |k|^4 |t-s|^2. $$
Combining the above estimates with \eqref{lem-estimate-1} we finally get the desired inequality.
\end{proof}

Using the above estimate and Cauchy's inequality,
  $$\aligned
  \E \big[\|\xi_N(t) - \xi_N(s)\|_{H^{-\beta-1}}^4 \big] &= \E \Bigg[\sum_{k\in \Z^2_0 } \frac{\<\xi_N(t) - \xi_N(s), e_k\>^2 } {|k|^{2(\beta +1)}} \Bigg]^2\\
  &\leq \Bigg[\sum_{k\in \Z^2_0 } \frac{1 } {|k|^{2(\beta +1)}} \Bigg] \Bigg[\sum_{k\in \Z^2_0 } \frac{\E \big( \<\xi_N(t) - \xi_N(s), e_k\>^4 \big) } {|k|^{2(\beta +1)}} \Bigg] \\
  &\leq  C|t-s|^2 \sum_{k\in \Z^2_0 } \frac1{|k|^{2(\beta +1) -8}} \leq C'|t-s|^2,
  \endaligned$$
since $\beta>4$. Consequently,
  $$\E \big[\|u_N(t) - u_N(s)\|_{H^{-\beta}}^4 \big] \leq C'|t-s|^2,$$
which implies
  $$\E \int_0^T\! \int_0^T \frac{\|u_N(t) - u_N(s)\|_{H^{-\beta}}^4}{|t-s|^{7/3}}\,\d t\d s \leq C .$$
Thus we have proved \eqref{tightness-estimate-1} and we obtain the tightness of $\{\eta_N\}_{N\geq 1}$ on $ C\big( [0,T]; H^{1-} \big)$. Equivalently, we have proved the tightness of the laws $\bar\eta_N$ of $\xi_N\, (N\geq 1)$ on $\mathcal X:= C\big( [0,T]; H^{-} \big)$.

Since we are dealing with the SDEs \eqref{SDE}, we need to consider $\bar\eta_N$ together with the laws of Brownian motions $\big\{ (W^k_t)_{0\leq t\leq T}: k\in \Z^2_0 \big\}$. To this end, we endow $\R^{\Z^2_0}$ with the metric
  $$d_\infty(a,b)= \sum_{k\in \Z^2_0} \frac{|a_k-b_k| \wedge 1}{2^{|k|}}, \quad a,b \in \R^{\Z^2_0}.$$
Then $\big( \R^{\Z^2_0}, d_\infty(\cdot, \cdot) \big)$ is separable and complete (see \cite[Example 1.2, p.9]{Billingsley}). The distance in $\mathcal Y:= C\big([0,T], \R^{\Z^2_0} \big)$ is given by
  $$d_{\mathcal Y}(w,\hat w) = \sup_{t\in [0,T]} d_\infty(w(t), \hat w(t)),\quad w, \hat w \in \mathcal Y,$$
which makes $\mathcal Y$ a Polish space. Denote by $\mathcal W$ the law on $\mathcal Y$ of the sequence of independent Brownian motions $\big\{ (W^k_t)_{0\leq t\leq T}: k\in \Z^2_0 \big\}$.

To simplify the notations, we write $W_\cdot= (W_t)_{0\leq t\leq T}$ for the whole sequence of processes $\big\{ (W^k_t)_{0\leq t\leq T}: k\in \Z^2_0 \big\}$ in $\mathcal Y$. For any $N\geq 1$, denote by $P_N$ the joint law of $(\xi_N(\cdot), W_\cdot )$ on
  $$\mathcal X \times \mathcal Y = C\big( [0,T]; H^{-} \big)\times C\big([0,T], \R^{\Z^2_0} \big).$$
Since the marginal laws $\{ \bar\eta_N \}_{N\in \N}$ and $\{\mathcal W\}$ are respectively tight on $\mathcal X$ and $\mathcal Y$, we conclude that $\{ P_N \}_{N\in \N}$ is tight on $\mathcal X \times \mathcal Y$. The Prohorov theorem (see \cite[Theorem 5.1, p.59]{Billingsley}) implies that there exists a subsequence $\{N_i\}_{i\in \N}$ such that $P_{N_i}$ converge weakly as $i\to \infty$ to some probability measure $P$ on $\mathcal X \times \mathcal Y$. By Skorokhod's representation theorem (\cite[Theorem 6.7, p.70]{Billingsley}), there exist a probability space $\big(\tilde\Omega, \tilde{\mathcal F}, \tilde \P \big)$, and stochastic processes $\big(\tilde \xi_{N_i}(\cdot), \tilde W^{N_i}_\cdot \big)_{i\in \N}$ and $\big(\tilde \xi(\cdot), \tilde W_\cdot \big)$ on this space with the corresponding laws $P_{N_i}$ and $P$ respectively, such that $\big(\tilde \xi_{N_i}(\cdot), \tilde W^{N_i}_\cdot \big)$ converge $\tilde\P$-a.s. in $\mathcal X\times \mathcal Y$ to the limit $\big(\tilde \xi(\cdot), \tilde W_\cdot \big)$. We are going to prove that $\big(\tilde \xi(\cdot), \tilde W_\cdot \big)$ is a weak solution to the equation \eqref{SEE-formulation}.

Denote by $\tilde u_{N_i} = K\ast \tilde \xi_{N_i}$ and $\tilde u = K\ast \tilde \xi$ which are the velocity fields defined on the new probability space $\big(\tilde\Omega, \tilde{\mathcal F}, \tilde \P \big)$. By the above discussions, we know that
  \begin{equation}\label{strong-convergence}
  \tilde\P \mbox{-a.s.}, \quad \tilde \xi_{N_i}(\cdot) \mbox{ converge strongly to } \tilde \xi(\cdot) \mbox{ in } C([0,T]; H^{-}),
  \end{equation}
which implies that
  \begin{equation}\label{strong-convergence-1}
  \tilde\P \mbox{-a.s.}, \quad \tilde u_{N_i}(\cdot) \mbox{ converge strongly to } \tilde u(\cdot) \mbox{ in } C([0,T]; H^{1-}).
  \end{equation}
The new processes $\tilde \xi_{N_i}(\cdot)$ (resp. $\tilde u_{N_i}(\cdot)$) have the same law with $\xi_{N_i}(\cdot)$ (resp. $u_{N_i}(\cdot)$), and thus by Lemma \ref{energy-estimate}, we have
  \begin{equation}\label{sect-3.1}
  \sup_{t\in [0,T]} \big\Vert \nabla^\perp\cdot \tilde u_{N_i}(t) \big\Vert_{L^2}
  = \sup_{t\in [0,T]} \big\Vert \tilde\xi_{N_i}(t) \big\Vert_{L^2}
  \leq \Vert \xi_0 \Vert_{L^2}\quad \tilde\P \mbox{-a.s.}
  \end{equation}

\begin{lemma}\label{lem-bddness}
The process $\tilde\xi$ has $\tilde{\mathbb{P}}$-a.s. weakly continuous trajectories in $L^2$ and satisfies
\begin{equation}\label{estim-final}
\sup_{t\in [0,T]} \Vert \tilde\xi(t)\Vert_{L^2}\leq \Vert \xi_0\Vert_{L^2}\quad \tilde{\mathbb{P}}\text{-a.s.}
\end{equation}
\end{lemma}

\begin{proof}
Thanks to \eqref{sect-3.1}, there exists a set $\Gamma\subset \tilde{\Omega}$ of full measure such that, for every $\omega\in \Gamma$, \eqref{strong-convergence} holds and
\begin{equation}\label{estim lem-bddness}
\sup_{i\geq 1} \sup_{t\in [0,T]} \Vert \tilde\xi_{N_i}(\omega, t)\Vert_{L^2} \leq \Vert \xi_0\Vert_{L^2}.
\end{equation}
Let us fix $\omega\in \Gamma$. Then by \eqref{estim lem-bddness} the sequence $\{\tilde\xi_{N_i} (\omega,\cdot)\}_{i\geq 1}$ is bounded in $L^\infty(0,T;L^2)$ and so we can extract a subsequence (not relabelled for simplicity) which is weak-$\ast$ convergent. But weak-$\ast$ convergence in $L^\infty(0,T;L^2)$ implies weak-$\ast$ convergence in $L^\infty(0,T;H^{-})$, which implies by \eqref{strong-convergence} that the limit is necessarily $\tilde\xi$; therefore by properties of weak-$\ast$ convergence
\begin{equation}\nonumber
\Vert \tilde\xi(\omega,\cdot)\Vert_{L^\infty(0,T;L^2)}\leq \liminf_N \Vert \tilde\xi_N(\omega,\cdot)\Vert_{L^\infty(0,T;L^2)}\leq \Vert \xi_0\Vert_{L^2}.
\end{equation}
In particular, there exists a subset $S_\omega\subset [0,T]$ of full Lebesgue measure (thus dense) such that $\Vert \tilde\xi(\omega,s) \Vert_{L^2}\leq \Vert\xi_0\Vert_{L^2}$ for every $s\in S_\omega$. Now let $t\in [0,T]\setminus S_\omega$ and consider a sequence $t_n\to t$, $t_n\in S_\omega$. Then the sequence $\{\tilde\xi(\omega, t_n)\}_n$ is uniformly bounded in $L^2$ and we can therefore extract a weakly convergent subsequence; but $\tilde\xi(\omega,\cdot)\in C([0,T];H^{-})$, therefore $\tilde\xi(\omega,t_n)\to \tilde\xi(\omega, t)$ in $H^{-}$ and so the weak limit must be $\tilde\xi(\omega,t)$. By properties of weak convergence we have
\begin{equation}\nonumber
\Vert \tilde\xi(\omega,t)\Vert_{L^2}\leq \liminf_n \Vert \tilde\xi(\omega,t_n)\Vert_{L^2}\leq \Vert \xi_0\Vert_{L^2}.
\end{equation}
As the reasoning holds for any $t\in [0,T]\setminus S_\omega$, for any $\omega\in\Gamma$, we have obtained
\begin{equation}\nonumber
\sup_{t\in [0,T]} \Vert \tilde\xi(\omega,t)\Vert_{L^2}\leq \Vert \xi_0\Vert_{L^2}\quad \forall\,\omega\in \Gamma,
\end{equation}
namely \eqref{estim-final}. It remains to show that, for every $\omega\in\Gamma$, $t\mapsto \tilde\xi(\omega,t)$ is weakly continuous in $L^2$. Let $t_n\to t$, then by \eqref{estim-final} the sequence $\{\tilde\xi(\omega,t_n)\}_n$ is bounded in $L^2$ and so it admits a weakly convergent subsequence. But $\tilde\xi(\omega,\cdot)\in C([0,T];H^{-})$, therefore the weak limit is necessarily $\tilde\xi (\omega,t)$; as the reasoning holds for any subsequence of $\{\tilde\xi(\omega,t_n)\}_n$, we deduce that the whole sequence is weakly converging to $\tilde\xi(\omega,t)$.
\end{proof}

Finally we can give the

\begin{proof}[Proof of Theorem \ref{thm-existence}]
The processes $\big(\tilde \xi_{N_i}(\cdot), \tilde W^{N_i}_\cdot \big)$ on the new probability space $\big(\tilde\Omega, \tilde{\mathcal F}, \tilde \P \big)$ have the same laws with that of $(\xi_{N_i}(\cdot), W_\cdot )$, which satisfy the equation \eqref{SDE} with $N$ replaced by $N_i$. Some classical arguments show that the stochastic integrals involved below make sense, see e.g. \cite[Section 2.6, p.89]{Krylov}. Therefore, for any $\phi\in C^\infty(\T^2)$, one has, $\tilde \P$-a.s for all $t\in [0,T]$,
  \begin{equation}\label{proof-0}
  \aligned
  \big\< \tilde\xi_{N_i}(t),\phi \big\> =&\, \big\< \xi_{N_i}(0),\phi \big\> +\int_0^t \big\< \tilde\xi_{N_i}(s), \tilde u_{N_i}(s)\cdot \nabla \phi \big\>\,\d s + \nu \int_0^t \big\<\tilde\xi_{N_i}(s), \Delta \phi \big\>\,\d s \\
  & - \eps \sum_{k\in \Z^2_0} \theta_k \int_0^t \big\<\tilde\xi_{N_i}(s), \sigma_k\cdot \nabla \phi \big\>\,\d\tilde W^{N_i,k}_s.
  \endaligned
  \end{equation}
We regard all the quantities as real valued stochastic processes. From the above discussions, we can prove that, as $i\to \infty$, all the terms of the first line converge in $L^1\big( \tilde\Omega, C([0,T], \R)\big)$ to the corresponding ones. Indeed, considering $\<\cdot, \cdot\>$ as the duality between distributions and smooth functions, then \eqref{strong-convergence} implies that, $\tilde\P$-a.s., $\big\< \tilde\xi_{N_i}(\cdot),\phi \big\>$ converge in $C([0,T], \R)$ to $\big\< \tilde\xi(\cdot),\phi \big\>$. Moreover, by \eqref{sect-3.1},
  $$\big| \big\< \tilde\xi_{N_i}(t),\phi \big\>\big| \leq \|\xi_0\|_{L^2} \|\phi\|_{L^2} \quad \tilde\P \mbox{-a.s. for all } t\in [0,T].$$
Thus the dominated convergence theorem implies the desired result. For the nonlinear term, we have
  $$\aligned
  &\ \E_{\tilde \P} \bigg[\sup_{t\in [0,T]}\bigg| \int_0^t \big\< \tilde\xi_{N_i}(s), \tilde u_{N_i}(s)\cdot \nabla \phi \big\>\,\d s - \int_0^t \big\< \tilde\xi(s), \tilde u(s)\cdot \nabla \phi \big\>\,\d s \bigg| \bigg] \\
  \leq &\ \E_{\tilde \P} \bigg[\sup_{t\in [0,T]}\bigg| \int_0^t \big\< \tilde\xi_{N_i}(s), \tilde u_{N_i}(s)\cdot \nabla \phi \big\>\,\d s - \int_0^t \big\< \tilde\xi_{N_i}(s), \tilde u(s)\cdot \nabla \phi \big\>\,\d s \bigg| \bigg]\\
  &\, + \E_{\tilde \P} \bigg[ \sup_{t\in [0,T]}\bigg| \int_0^t \big\< \tilde\xi_{N_i}(s), \tilde u(s)\cdot \nabla \phi \big\>\,\d s - \int_0^t \big\< \tilde\xi(s), \tilde u(s)\cdot \nabla \phi \big\>\,\d s \bigg| \bigg].
  \endaligned$$
By \eqref{sect-3.1}, the sequence $\tilde u_{N_i}$ is almost surely bounded in $L^2(0,T; H)$. This plus the almost sure convergence \eqref{strong-convergence-1} of $\tilde u_{N_i}$ to $\tilde u$ in $C([0,T]; H^{1-})$ implies that $\tilde u_{N_i}$ converge strongly in $L^2\big(\tilde\Omega, L^2(0,T; H) \big)$ to  $\tilde u$. Thanks to \eqref{sect-3.1} and \eqref{estim-final}, the first term on the right hand side vanishes as $i\to \infty$. For the second term, by \eqref{strong-convergence}, the quantity in the square bracket tends to 0 $\tilde\P$-a.s., which together with the bounds \eqref{sect-3.1} and \eqref{estim-final}, the dominated convergence theorem leads to the desired result.

It remains to show the convergence of the stochastic integrals. Fix any $M\in \N$; we have
  \begin{equation}\label{proof-1}
  \aligned
  &\ \E_{\tilde \P} \bigg[\sup_{t\in [0,T]}\bigg| \sum_{k\in \Z^2_0} \theta_k \int_0^t \big\<\tilde\xi_{N_i}(s), \sigma_k\cdot \nabla \phi \big\>\,\d \tilde W^{N_i,k}_s - \sum_{k\in \Z^2_0} \theta_k \int_0^t \big\<\tilde\xi(s), \sigma_k\cdot \nabla \phi \big\>\,\d\tilde W^{k}_s \bigg| \bigg] \\
  \leq &\ \E_{\tilde \P} \bigg[ \sup_{t\in [0,T]}\bigg| \sum_{|k|\leq M} \theta_k \bigg(\int_0^t \big\<\tilde\xi_{N_i}(s), \sigma_k\cdot \nabla \phi \big\>\,\d \tilde W^{N_i,k}_s - \int_0^t \big\<\tilde\xi(s), \sigma_k\cdot \nabla \phi \big\>\,\d\tilde W^{k}_s \bigg) \bigg| \bigg] \\
  &\, + \E_{\tilde \P} \bigg[ \sup_{t\in [0,T]}\bigg| \sum_{|k|> M} \theta_k \int_0^t \big\<\tilde\xi_{N_i}(s), \sigma_k\cdot \nabla \phi \big\>\,\d \tilde W^{N_i,k}_s \bigg| \bigg] \\
  &\, + \E_{\tilde \P} \bigg[ \sup_{t\in [0,T]}\bigg| \sum_{|k|> M} \theta_k \int_0^t \big\<\tilde\xi(s), \sigma_k\cdot \nabla \phi \big\>\,\d \tilde W^{k}_s \bigg| \bigg].
  \endaligned
  \end{equation}
We denote the three expectations on the right hand side by $J_{N_i}^{(n)},\, n=1,2,3$. First,
  $$\aligned
  \big| J_{N_i}^{(2)} \big| &\leq C\, \E_{\tilde \P} \bigg[ \bigg( \sum_{|k|> M} \theta_k^2 \int_0^T \big\<\tilde\xi_{N_i}(s), \sigma_k\cdot \nabla \phi \big\>^2\,\d s \bigg)^{1/2} \bigg] \\
  &\leq C\|\theta\|_{\ell^\infty_{>M}}\, \E_{\tilde \P} \bigg[ \bigg( \sum_{|k|> M} \int_0^T \big\<\tilde\xi_{N_i}(s)\nabla \phi , \sigma_k \big\>^2 \,\d s \bigg)^{1/2} \bigg] \\
  &\leq C\|\theta\|_{\ell^\infty_{>M}}\, T^{1/2} \|\xi_0\|_{L^2} \|\nabla\phi \|_\infty,
  \endaligned$$
where $\|\theta\|_{\ell^\infty_{>M}} = \sup_{|k|>M} |\theta_k|$ tends to 0 as $M\to \infty$, due to $\theta\in \ell^2$. Similar estimate holds for $J_{N_i}^{(3)}$ by Lemma \ref{lem-bddness}.

Finally, we deal with $J_{N_i}^{(1)}$ for which we need Skorohod's result for convergence of stochastic integrals, see for instance \cite[Lemma 5.2]{GyMa} and \cite[Lemma 3.2]{Luo} for a slightly more general version. By the discussions above Lemma \ref{lem-bddness}, we known that as $i\to \infty$, $\tilde\P$-a.s. for all $s\in [0,T]$, $\big\<\tilde\xi_{N_i}(s), \sigma_k\cdot \nabla \phi \big\> \to \big\<\tilde\xi(s), \sigma_k\cdot \nabla \phi \big\>$ and $\tilde W^{N_i,k}_s \to \tilde W^{k}_s$. Since there are only finitely many stochastic integrals, by  \cite[Lemma 3.2]{Luo}, it is sufficient to show that, for any $|k|\leq M$,
  $$\bigg(\E_{\tilde \P} \int_0^T \big\<\tilde\xi(s), \sigma_k\cdot \nabla \phi \big\>^4 \,\d s \bigg) \bigvee \bigg(\sup_{i\geq 1} \E_{\tilde \P} \int_0^T \big\<\tilde\xi_{N_i}(s), \sigma_k\cdot \nabla \phi \big\>^4 \,\d s \bigg)<+\infty.$$
Indeed, by Lemma \ref{lem-bddness},
  $$\E_{\tilde \P} \int_0^T \big\<\tilde\xi(s), \sigma_k\cdot \nabla \phi \big\>^4 \,\d s \leq \E_{\tilde \P} \int_0^T \big\|\tilde\xi(s) \big\|_{L^2}^4 \big\|\sigma_k\cdot \nabla \phi \big\|_{L^2}^4  \,\d s \leq T \|\xi_0\|_{L^2}^4 \|\nabla\phi\|_\infty^4. $$
Analogous uniform estimate holds for the second part. Therefore we obtain $\lim_{i\to \infty} J_{N_i}^{(1)} =0$. First letting $i\to \infty$ and then $M\to \infty$ in \eqref{proof-1}, we have proved the convergence of stochastic integrals.

Therefore, letting $i\to \infty$ in \eqref{proof-0}, we obtain, $\tilde\P$-a.s. for all $t\in [0,T]$,
  $$\aligned
  \big\< \tilde\xi(t),\phi \big\> =&\, \big\< \xi(0),\phi \big\> +\int_0^t \big\< \tilde\xi(s), \tilde u(s)\cdot \nabla \phi \big\>\,\d s + \nu \int_0^t \big\<\tilde\xi(s), \Delta \phi \big\>\,\d s \\
  & - \eps \sum_{k\in \Z^2_0} \theta_k \int_0^t \big\<\tilde\xi(s), \sigma_k\cdot \nabla \phi \big\>\,\d\tilde W^{k}_s.
  \endaligned $$
This completes the proof.
\end{proof}

\section{Convergence to 2D Navier--Stokes equations} \label{sec4}

In this section we show that the solutions to \eqref{SEE-approx-formulation} converge weakly to the unique solution of the deterministic 2D Navier--Stokes equations.

Let us briefly recall the setting: we fix $\xi_0\in L^2$ and $\nu>0$, we consider a sequence $\big\{\theta^N_\cdot \big\}_{N\geq 1}$ satisfying \eqref{theta} and \eqref{theta-N}, and define $\varepsilon_N$ by \eqref{eps-N}. For each $N$, we consider a weak solution $\xi^N$ of \eqref{SEE-approx-formulation} with initial data $\xi_0$, satisfying
  \begin{equation}\label{sect-4.1}
  \sup_{t\in [0,T]} \big\|\xi^N_t \big\|_{L^2} \leq \|\xi_0 \|_{L^2} \quad \P\mbox{-a.s.},
  \end{equation}
whose existence is granted by Theorem \ref{thm-existence}. Since we are dealing with weak solutions, the processes $\xi^N$ might be defined on different probability space; however, for the sake of simplicity, in the following we do not distinguish the notations $\E$, $\P$, $\Omega$, etc.

Let us immediately remark that conditions \eqref{theta-N} and \eqref{eps-N} together imply
$$ \lim_{N\to\infty} \varepsilon_N \big\Vert \theta^N_\cdot \big\Vert_{\ell^\infty} =0,$$
therefore the sequence $\big\{\varepsilon_N \big\Vert \theta^N_\cdot \big\Vert_{\ell^\infty} \big\}_{N\geq 1}$ is bounded by a suitable constant.

Let $Q^N$ denote the law of $\xi^N$, $N\geq 1$. Similarly as in Section 3, we can show that $\{Q^N\}_{N\geq 1}$ is tight on $C([0,T];H^-)$, which can be reduced to show that it is tight on $C\big([0,T];H^{-\delta} \big)$ for any $\delta\in (0,1)$. We sketch the proof here. First, similar to Theorem \ref{thm-simon}, we have the following result: given $\beta>4$, if $p> 12 (1+\beta -\delta)/\delta$, then
  $$L^p(0,T; H) \cap W^{1/3, 4} \big(0,T; H^{-1-\beta} \big) \subset C\big( [0,T]; H^{-\delta} \big) $$
is a  compact embedding. Thus, by \eqref{sect-4.1}, to prove the tightness of $\{Q^N\}_{N\geq 1}$ on $C\big([0,T];H^{-\delta} \big)$, it is enough to show that
\begin{equation}\label{sect-4-bound}
\sup_{N\geq 1} \mathbb{E}\int_0^T\! \int_0^T \frac{\big\Vert \xi^N_t-\xi^N_s \big\Vert_{H^{-1-\beta}}^4}{\vert t-s\vert^{7/3}}\,\d t\d s <\infty.
\end{equation}
To this aim, it suffices to obtain estimates similar to those of Lemma \ref{lem-estimate}, taking care that all the constants involved do not depend on $\theta^N_\cdot$ nor $\varepsilon_N$.

\begin{lemma}\label{lem-estimate 2}
There is a constant $C>0$ such that for any $N\geq 1$, $0\leq s<t \leq T$,
  $$\E \big( \big\<\xi^N_t - \xi^N_s, e_k \big\>^4 \big) \leq C |k|^8 |t-s|^2 \quad \mbox{for all } k\in \Z^2_0.$$
\end{lemma}

\begin{proof} For any  fixed $k$, since $\xi^N$ is a solution of \eqref{SEE-approx-formulation}, it holds
\begin{equation}\nonumber\begin{split}
 \big\<\xi^N_t-\xi^N_s,e_k \big\> = \int_s^t \big\<\xi^N_r, u^N_r\cdot \nabla e_k \big\>\,\d r + \nu \int_s^t \big\<\xi^N_r, \Delta e_k \big\> \,\d r - \eps_N \sum_{l\in \Z^2_0} \theta^N_l \int_s^t \big\<\xi^N_r, \sigma_l\cdot \nabla e_k \big\>\,\d W^l_r.
\end{split}\end{equation}
Using this equation and the bound \eqref{sect-4.1}, we can proceed in the same way as the proof of Lemma \ref{lem-estimate}; we omit it here.
\end{proof}

Thanks to Lemma \ref{lem-estimate 2}, an analogous computation below the proof of Lemma \ref{lem-estimate} gives us the uniform estimate \eqref{sect-4-bound}. As a result, we conclude that the family $\big\{Q^N \big\}_{N\geq 1}$ is tight in $C\big([0,T];H^{-\delta} \big)$.

With the above preparations, the proof of Theorem \ref{thm-main} is similar to that of Theorem \ref{thm-existence}. However, we would like to provide here a slightly different argument, without using the Skorohod representation theorem. First, by the estimate \eqref{sect-4.1} we know that, for all $N$, almost every realization of $\xi^N$ satisfies
\begin{equation}\nonumber
\int_0^T \big\Vert \xi^N_r \big\Vert_{L^2}^2\,\d r\leq T\Vert\xi_0\Vert_{L^2}^2.
\end{equation}
In particular, if we fix a radius $R\geq \sqrt{T}\Vert \xi_0\Vert_{L^2}$ and consider the space
  \begin{equation} \label{weak-space}
  L^2_{R,w} = \big\{ f\in L^2(0,T;H) : \Vert f\Vert_{L^2(0,T;H)}\leq R \big\}
  \end{equation}
endowed with the weak topology, then it is a metrizable, compact space (see for instance \cite{Bre}); we can regard $\big\{ \xi^N \big\}_{N\geq 1}$ as random variables taking values in $L^2_{R,w}$ and so by compactness their laws form a tight sequence in such space. Next, note that the tightness of $\big\{Q^N \big\}_{N\geq 1}$ in $C\big([0,T];H^{-\delta} \big)$ implies that the family  $\big\{Q^N \big\}_{N\geq 1}$ is also tight on $L^2\big(0,T;H^{-1} \big)$. As a result, $\big\{Q^N \big\}_{N\geq 1}$ is tight in $L^2\big(0,T;H^{-1} \big)\cap L^2_{R,w}$.

Before giving the proof of the second part of Theorem \ref{thm-main}, we need the following lemma.

\begin{lemma}\label{lem continuity sec 4}
For any $\phi\in C^\infty(\mathbb{T}^2)$, consider the map
\begin{equation}\nonumber
F_\phi (f)_{\cdot} = \langle f_\cdot,\phi\rangle - \langle \xi_0,\phi\rangle -\int_0^\cdot \langle (K\ast f_s)\cdot\nabla\phi, f_s\rangle\,\d s - \nu \int_0^\cdot \<f_s, \Delta\phi\>\,\d s.
\end{equation}
Then $F_\phi$ is a continuous bounded map from $L^2\big(0,T;H^{-1} \big)\cap L^2_{R,w}$ into $L^2(0,T;\mathbb{R})$.
\end{lemma}

\begin{proof}
Let us show boundedness first. We have
\begin{equation}\nonumber\begin{split}
\vert F_\phi(f)_t\vert
& \leq \Vert f_t\Vert_{L^2}\,\Vert \phi\Vert_{L^2} + \Vert \xi_0\Vert_{L^2}\,\Vert \phi\Vert_{L^2} + \int_0^t \vert \langle (K\ast f_s)\cdot\nabla\phi, f_s\rangle\vert\,\d s + \nu \int_0^t |\<f_s, \Delta\phi\>| \,\d s \\
& \leq \Vert f_t\Vert_{L^2}\,\Vert \phi\Vert_{L^2} + \Vert \xi_0\Vert_{L^2}\,\Vert \phi\Vert_{L^2} + \Vert \nabla\phi\Vert_{L^\infty} \int_0^T \Vert f_s\Vert_{L^2}^2\, \d s + \nu \|\Delta\phi\|_\infty \int_0^T \|f_s\|_{L^2} \,\d s\\
& \leq \Vert \phi\Vert_{C^2} (\Vert f_t\Vert_{L^2} + \Vert \xi_0\Vert_{L^2} + C_{R,T}),
\end{split}\end{equation}
where we used the fact that $f\in L^2_{w,R}$, and $C_{R,T}$ is a constant depending on $R$ and $T$. Therefore
$$\Vert F_\phi(f)\Vert_{L^2(0,T;\mathbb{R})}
\leq \Vert \phi\Vert_{C^2} \big(\Vert f\Vert_{L^2(0,T;L^2)} + \Vert \xi_0\Vert_{L^2} + C_{R,T} \big)
\leq \Vert \phi\Vert_{C^2} \big(\Vert \xi_0\Vert_{L^2} + C'_{R,T} \big).$$
Regarding continuity: let $f^n$ be a sequence converging to $f$ in $L^2\big(0,T;H^{-1} \big)\cap L^2_{R,w}$, namely $f^n \to f$ strongly in $L^2\big(0,T;H^{-1} \big)$ and weakly in $L^2\big(0,T;L^2 \big)$. Strong convergence in $L^2\big(0,T;H^{-1} \big)$ implies convergence of $\langle f^n,\phi\rangle$ to $\langle f,\phi\rangle$ in $L^2(0,T)$, similarly for $\int_0^\cdot \langle f^n,\Delta\phi\rangle\,\d s$ to $\int_0^\cdot \langle f,\Delta\phi\rangle\,\d s$; so we only need to check convergence of the nonlinear term. By properties of the Biot--Savart kernel, $K\ast f^n\to K\ast f$ strongly in $L^2\big(0,T;L^2 \big)$; combining the strong convergence of $K\ast f^n$ and the weak convergence of $f^n$ we obtain, that for any $t\in (0,T)$,
\begin{equation}\nonumber
\int_0^t \big\langle (K\ast f^n_s)\cdot\nabla\phi, f^n_s \big\rangle\,\d s \to \int_0^t \big\langle (K\ast f_s)\cdot\nabla\phi, f_s \big\rangle \,\d s.
\end{equation}
Therefore pointwise convergence holds; the previous estimates also show uniform boundedness of the integral processes, therefore by dominated convergence we obtain the conclusion.
\end{proof}

Finally we can complete the

\begin{proof}[Proof of Theorem \ref{thm-main}]
The fact that $\xi^N$ are solutions of \eqref{SEE-approx-formulation} may be formulated as follows: for every $\phi\in C^\infty(\mathbb{T}^2)$, the equality $F_\phi\big(\xi^N \big)=M^N_\phi$ holds, where $F_\phi$ is defined as in Lemma \ref{lem continuity sec 4} and $M^N_\phi$ is the process given by
\begin{equation}\nonumber
M^N_\phi = -\varepsilon_N\sum_{k\in\mathbb{Z}^2_0} \theta^N_k \int_0^\cdot \big\< \xi^N_s,\sigma_k\cdot\nabla\phi \big\>\,\d W^k_s.
\end{equation}
The sequence $\big\{Q^N \big\}_{N\geq 1}$ is tight in $L^2\big(0,T;H^{-1} \big)\cap L^2_{R,w}$, therefore by Prohorov theorem we can extract a subsequence (not relabelled for simplicity) which is weakly converging to the law $Q$ of some $L^2\big(0,T;H^{-1} \big)\cap L^2_{R,w}$-valued random variable $\xi$. By Lemma \ref{lem continuity sec 4}, $F_\phi$ is a continuous and bounded map, therefore by properties of convergence in law $F_\phi(\xi^N)$ are also converging in distribution to $F_\phi(\xi)$; in particular this implies that $M^N_\phi$ are also converging to some limit.
On the other side, by It\^o's isometry we have
\begin{equation}\nonumber\begin{split}
\mathbb{E}\int_0^T \big\vert M^N_\phi(t) \big\vert^2\,\d t
& = \varepsilon_N^2\int_0^T \mathbb{E}\int_0^t \sum_{k\in\mathbb{Z}_0^2} \big(\theta_k^N \big)^2 \big\<\xi^N_s,\sigma_k\cdot \nabla\phi \big\>^2\, \d s\, \d t\\
& \leq T \varepsilon_N^2 \big\Vert \theta^N_\cdot \big\Vert_{\ell^\infty}^2 \mathbb{E}\int_0^T\sum_{k\in\mathbb{Z}_0^2} \big\< \xi^N_s\,\nabla\phi, \sigma_k \big\>^2\,\d s\\
& \leq T \varepsilon_N^2 \big\Vert \theta^N_\cdot \big\Vert_{\ell^\infty}^2 \int_0^T \mathbb{E}\Big( \big\Vert \xi_s^N\,\nabla\phi \big\Vert_{L^2}^2 \Big)\,\d s\\
& \leq T^2\Vert \xi_0\Vert^2_{L^2}\,\Vert\nabla\phi\Vert_{L^\infty}^2\,\varepsilon_N^2 \big\Vert \theta^N_\cdot \big\Vert_{\ell^\infty}^2 \to 0 \quad \text{ as }N\to\infty
\end{split}\end{equation}
which implies that $M^N_\phi$ is converging in law to 0; therefore $F_\phi(\xi)=0$, up to a $Q$-negligible set. Given a countable dense set $\{\phi_n\}_n$, we can deduce that the support of $Q$ satisfies $F_{\phi_n}(\xi)=0$ for all $n$. This, together with its $L^2$-boundedness, implies that $F_\phi(\xi)=0$ for all $\phi$. Namely, the support of $Q$ is made of solutions of the deterministic 2D Navier--Stokes equation \eqref{NSE} starting at $\xi_0$; therefore by uniqueness $Q$ is given by $\delta_\xi$, where $\xi$ is such unique solution. As the reasoning applies to any subsequence of $\big\{Q^N \big\}_{N\geq 1}$, we deduce convergence in law of the whole sequence to $\delta_\xi$.
\end{proof}

As a consequence of Theorem \ref{thm-main} we deduce convergence of the passive scalars advected by $u^N$ to those advected by $u$, where as usual $u^N$ and $u$ denote the velocity fields associated to $\xi^N$ and $\xi$. To state the result, we assume for simplicity the sequence $u^N$ to be defined on the same filtered probability space $(\Omega,\mathcal{F},\mathcal{F}_t,\mathbb{P})$ and such that $u^N(\omega)\to u$ in $L^2\big(0,T;L^2 \big)$ for every $\omega\in \Gamma$, a set of full probability; this comes without loss of generality by applying Skorokhod's theorem. For a given $\rho_0\in L^p(\mathbb{T}^2)$, $p\in (1,\infty)$, we denote by $\rho^N$ the passive scalar advected by $u^N$ with initial configuration $\rho_0$, i.e. the solution of
\begin{equation}\label{passive scalar eq sec 4}
\begin{cases}
\partial_t \rho^N + u^N\cdot\nabla\rho^N = 0,\\ \rho^N(0)=\rho_0;
\end{cases}
\end{equation}
similarly for $\rho$ and $u$. By \eqref{energy estimate thm-existence}, we can take $\Gamma$ such that $\sup_{N\geq 1} \big\Vert u^N(\omega) \big\Vert_{L^\infty(0,T; H^1)}\leq \Vert \xi_0\Vert_{L^2}$ for every $\omega\in \Gamma$ and thus, by the DiPerna--Lions theory, equation \eqref{passive scalar eq sec 4} admits a unique weak solution, which belongs to $C([0,T];L^p)$; similarly for $\rho$. We have the following

\begin{corollary}\label{corollary sec4}
For any $\omega\in \Gamma$, any $p\in(1,+\infty)$ and any $\rho_0\in L^p$, $\rho^N(\omega)\to \rho$ in $C([0,T];L^p)$.
\end{corollary}

\begin{proof} It follows immediately from \cite[Theorem II.5, p. 527]{DiPLio}.
\end{proof}

\section{Convergence of finite dimensional approximations}\label{sec5}

The setting of this section is the same as Section \ref{sec4} in terms of $\xi_0$, $\nu$, $\big\{\theta^N_\cdot \big\}_N$ and $\varepsilon_N$. However, for any $N$ we now consider $\xi^N$ to be an $H_N$-valued solution of the following SDE:

\begin{equation}\label{SDE-sec5}
\d \xi^N = -b_N\big(\xi^N \big)\,\d t + \varepsilon_N \sum_{k\in\mathbb{Z}^2_0}\theta^N_k G_N^k\big( \xi^N\big)\circ\d W^k, \quad \xi^N_0=\Pi_N \xi_0,
\end{equation}
where the vector fields $b_N$ and $ G_N^k$ are defined at the beginning of Section \ref{sec3}. Recall that $G_N^k\big( \xi^N\big)=0$ whenever $\vert k\vert>2N$, thus the series appearing on the right hand side is finite. We are interested in determining conditions on $\big\{ \theta^N_\cdot \big\}_N$ under which $\xi^N$ converge in law to the unique solution of \eqref{NSE}. Different finite dimensional schemes, like \eqref{SDE}, can also be considered; here we use \eqref{SDE-sec5} in order to show that the method is fairly robust and does not depend directly on the nature of the system, \eqref{SDE} being dissipative while \eqref{SDE-sec5} being conservative. The additional difficulty with respect to the previous sections is that the It\^o--Stratonovich corrector is not exactly $\nu\Delta\xi$, but is dependent of the finite-dimensional approximation, therefore we need to take care of its convergence in the limit.

\begin{lemma}
Equation \eqref{SDE-sec5} admits a unique strong solution $\xi^N$, satisfying
\begin{equation}\label{energy conservation sec5}
\mathbb{P}\big( \big\Vert \xi^N_t \big\Vert_{L^2} = \big\Vert \xi^N_0 \big\Vert_{L^2} \mbox{ for all } t\in [0,T]\big)=1.
\end{equation}
\end{lemma}

\begin{proof} All vector fields in \eqref{SDE-sec5} are smooth and $H_N$ is finite dimensional, so local existence and uniqueness follows. By Stratonovich chain rule,
\begin{equation}\nonumber
\d \Big( \frac{1}{2}\Vert \xi^N\Vert^2_{L^2}\Big)
= \langle \xi^N,\circ\, \d \xi^N\rangle
= -\big\langle b_N\big(\xi^N \big),\xi^N \big\rangle\,\d t + \varepsilon_N\sum_{k\in\mathbb{Z}^2_0} \theta_k^N \big\langle G_N^k\big( \xi^N\big) ,\xi^N \big\rangle\circ \d W^k = 0
\end{equation}
where the last equality follows from \eqref{properties}. This shows that $\Vert\cdot\Vert_{L^2}$ is invariant and implies global existence as well as the last statement.
\end{proof}

Next, it is clear that $\xi^N$ is a solution of \eqref{SDE-sec5} if and only if, for any $\phi\in H_N$, one has
\begin{equation}\nonumber
\d \big\langle\xi^N,\phi \big\rangle
= -\big\langle b_N\big(\xi^N \big),\phi \big\rangle\,\d t +\varepsilon_N\sum_{k\in\mathbb{Z}^2_0} \theta_k^N \big\langle G_N^k \big(\xi^N \big),\phi \big\rangle\circ\d W^k.
\end{equation}
Integration by parts and properties of the Stratonovich integral then yield
\begin{equation}\nonumber\begin{split}
\d\big\langle \xi^N,\phi \big\rangle
& = \big\langle \big(K\ast\xi^N \big)\cdot\nabla\phi,\xi^N \big\rangle\,\d t -\varepsilon_N\sum_{k\in\mathbb{Z}^2_0} \theta^N_k \big\langle \Pi_N(\sigma_k\cdot\nabla\phi),\xi^N \big\rangle\circ\d W^k\\
& = \big\langle \big(K\ast\xi^N \big)\cdot\nabla\phi,\xi^N \big\rangle\,\d t -\varepsilon_N\sum_{k\in\mathbb{Z}^2_0} \theta^N_k \big\langle \sigma_k\cdot\nabla\phi,\xi^N \big\rangle\,\d W^k + \frac{\varepsilon_N^2}{2} \big\langle C_N\phi,\xi^N \big\rangle\,\d t,
\end{split}\end{equation}
where $C_N$ is given by
\begin{equation}\label{def-C-N}
C_N\phi = \sum_{k\in\mathbb{Z}^2_0} \big(\theta_k^N \big)^2\, \sigma_k\cdot\nabla(\Pi_N(\sigma_k\cdot\nabla\phi)).
\end{equation}
Recall that for fixed $N$, the sum over $k$ has a finite amount of non zero terms, so all the above calculations (and the following) are rigorous.

It remains to compute $C_N$ explicitly, for which we need to introduce some notation. Recall that $\Pi_N$ is the orthogonal projection on $H_N$; with a slight abuse we identify it with the associated convolution kernel: $\Pi_N\xi=\Pi_N\ast \xi$. We denote the scalar product between matrices by $A:B=\text{Tr}(A^T B)$. For fixed $N$, let us define
\begin{equation}\label{covariance-matrix}
A_N(x,y):= \sum_{k\in\mathbb{Z}^2_0} \big(\theta^N_k \big)^2\,\sigma_k(x)\otimes\sigma_k(y),
\end{equation}
which is the covariance operator associated to the noise
\begin{equation}\nonumber
W_N(t,x)=\sum_{k\in\mathbb{Z}^2_0} \theta^N_k\,\sigma_k(x)\,W^k(t).
\end{equation}
It is easy to check that $A_N$ is homogeneous and it holds
\begin{equation}\nonumber
A_N(x,y)=A_N(x-y)= \sqrt{2} \sum_{k\in\mathbb{Z}^2_+} \big(\theta_k^N \big)^2\frac{k^\perp\otimes k^\perp}{\vert k\vert^2}\, e_k(x-y);
\end{equation}
in particular, identity \eqref{useful-equality} can be rewritten as
\begin{equation}\label{useful-equal}
A_N(x, x)=A_N(0)=\frac{1}{2} \big\Vert \theta^N_\cdot \big\Vert_{\ell^2}^2 I_2.
\end{equation}
Moreover, $A_N$ has Fourier transform given by
\begin{equation}\nonumber
\hat{A}_N(k) = \sqrt{2}\, \big(\theta_k^N \big)^2 \frac{k^\perp\otimes k^\perp}{\vert k\vert^2} {\bf 1}_{\{k\in \Z^2_+\}} ,
\end{equation}
which implies
\begin{equation}\nonumber
\big\Vert \hat{A}_N \big\Vert_{\ell^1} = \sqrt{2} \sum_{k\in\mathbb{Z}^2_+} \big(\theta^N_k \big)^2 = \frac{\sqrt{2}}2 \big\Vert \theta^N_\cdot \big\Vert_{\ell^2}^2.
\end{equation}
Recall the definition of the operator $C_N$ in \eqref{def-C-N}. Now we can prove

\begin{proposition}\label{fundamental lemma sec5}
It holds that
\begin{equation}\nonumber
C_N\phi(x) = (\Pi_N A_N)\ast \nabla^2\phi\, (x) = \int_{\mathbb{T}^2} \Pi_N(x-y)A_N(x-y): \nabla^2\phi(y)\,\d y
\end{equation}
and
\begin{equation}\nonumber
\Vert C_N\phi\Vert_{L^2} \leq \Vert\theta^N_\cdot\Vert_{\ell^2}^2\,\Vert \nabla^2\phi\Vert_{L^2}.
\end{equation}
Moreover, under condition \eqref{condition thm5}, for any $\phi\in C^\infty(\T^2)$,
  \begin{equation}\label{prop-key-limit}
  \lim_{N\to \infty} \frac{\eps_N^2}2 C_N\phi = \nu \Delta\phi \quad \mbox{holds in } L^2(\T^2).
  \end{equation}
\end{proposition}

\begin{proof}
Using the fact that $\Pi_N$ and $\nabla$ commute, we have
\begin{equation}\nonumber\begin{split}
C_N\phi(x)
& = \sum_{k\in\mathbb{Z}^2_0} \big(\theta_k^N \big)^2\, \sigma_k(x)\cdot\Pi_N[\nabla(\sigma_k\cdot\nabla\phi))](x)\\
& = \sum_{k\in\mathbb{Z}^2_0} \big(\theta_k^N \big)^2\, \int_{\mathbb{T}^2}\Pi_N(x-y)\,\sigma_k(x)\cdot\nabla(\sigma_k\cdot\nabla\phi)(y)\,\d y.
\end{split}\end{equation}
Note that $\sigma_k(x)\cdot\nabla\sigma_k(y)=0$ for all $k$, $x$ and $y$, thus by \eqref{covariance-matrix},
\begin{equation}\nonumber\begin{split}
C_N\phi(x)
& = \sum_{k\in\mathbb{Z}^2_0} \big(\theta_k^N \big)^2\, \int_{\mathbb{T}^2}\Pi_N(x-y)\,\sigma_k(x)\otimes \sigma_k(y) : \nabla^2\phi(y)\,\d y\\
& = \int_{\mathbb{T}^2} \Pi_N(x-y)A_N(x-y) : \nabla^2\phi(y)\,\d y.
\end{split}\end{equation}
Next, by Parseval identity and Young inequality we have
\begin{equation}\nonumber\begin{split}
\Vert C_N\phi\Vert_{L^2}
& = \big\Vert (\Pi_N A_N)\ast \nabla^2\phi \big\Vert_{L^2}
= \big\Vert \big(\hat{\Pi}_N\ast\hat{A}_N \big) \widehat{\nabla^2\phi} \big\Vert_{\ell^2}
\leq \big\Vert \hat{\Pi}_N\ast\hat{A}_N \big\Vert_{\ell^\infty}\, \big\Vert \widehat{\nabla^2\phi} \big\Vert_{\ell^2}\\
& \leq \big\Vert \hat{\Pi}_N \big\Vert_{\ell^\infty} \big\Vert \hat{A}_N \big\Vert_{\ell^1} \big\Vert \nabla^2\phi \big\Vert_{L^2}
\leq \big\Vert \theta^N_\cdot \big\Vert_{\ell^2}^2 \big\Vert \nabla^2\phi \big\Vert_{L^2}.
\end{split}\end{equation}

To show the last assertion, let $\Pi_N^\perp$ denote the orthogonal projection on $H_N^\perp$, which, with a slight abuse of notation, is identified with the associated convolution kernel. In this way, $\Pi_N+\Pi_N^\perp=I$ in the sense of linear operators on $L^2$ and $\Pi_N+\Pi_N^\perp = \delta $ in the sense of convolution with a distribution. Then for any fixed $N$ and any $\phi$ smooth, by \eqref{useful-equal}, it holds
\begin{equation}\nonumber\begin{split}
\nu\Delta\phi(x)
& = \frac{\varepsilon_N^2}{2} \int_{\mathbb{T}^2} A_N(x-y): \nabla^2\phi(y)\, \delta(\d y)\\
& = \frac{\varepsilon_N^2}{2} \bigg[ \int_{\mathbb{T}^2} \Pi_N(x-y)A_N(x-y): \nabla^2\phi(y)\,\d y
+ \int_{\mathbb{T}^2} \Pi^\perp_N(x-y)A_N(x-y): \nabla^2\phi(y)\,\d y \bigg] \\
& =: \frac{\varepsilon_N^2}{2} C_N\phi + \frac{\varepsilon_N^2}{2} C_N^\perp \phi.
\end{split}\end{equation}
Assertion \eqref{prop-key-limit} then is equivalent to showing that $\varepsilon_N^2 C_N^\perp \phi\to 0$ as $N\to\infty$. By an approximation argument, we can take $\phi$ to be a finite linear combination of $e^{-i 2\pi j\cdot x},\, j\in \mathbb{Z}_0^2$. In this case, it is enough to prove that, for any $j\in \Z_0^2$, $\varepsilon_N^2 C_N^\perp e^{-i 2\pi j\cdot x} \to 0$ as $N\to \infty$. We have
\begin{equation}\nonumber\begin{split}
\varepsilon_N^2 \big\Vert C_N^\perp e^{-i 2\pi j\cdot x} \big\Vert_{L^2}
= 4\pi^2 \varepsilon_N^2 \Big\vert \widehat{A_N \Pi_N^\perp}(j) : (j\otimes j)\Big\vert
\leq K|j|^2 \big\Vert \theta^N_\cdot \big\Vert_{\ell^2}^{-2} \sum_{k:\vert k-j\vert>N} \big(\theta_k^N \big)^2.
\end{split}\end{equation}
This shows that, under condition \eqref{condition thm5}, claim \eqref{prop-key-limit} holds and the proof is complete.
\end{proof}

It follows from Proposition \ref{fundamental lemma sec5} and our choice \eqref{eps-N} of $\varepsilon_N$ that, for any $\phi\in H_N$, one has
\begin{equation}\label{estimate approx. corrector sec5}
\frac{\varepsilon_N^2}{2} \Vert C_N\phi\Vert_{L^2} \leq 2\nu \big\Vert \nabla^2\phi \big\Vert_{L^2}.
\end{equation}
This provides a uniform control on the correctors $C_N$ for $N\in \N$. Let $Q^N$ denote the law of $\xi^N$, solution to \eqref{SDE-sec5}, then we can prove the following:

\begin{lemma}
The family $\big\{ Q^N \big\}_N$ is tight in $C\big([0,T];H^-(\mathbb{T}^2) \big).$
\end{lemma}

\begin{proof}
We only sketch the proof briefly since most of the calculations are identical to those of Section \ref{sec3}. Indeed by the energy equality \eqref{energy conservation sec5} and Theorem \ref{thm-simon}, we only need to show that there exists a constant $C$ such that, for any $N\geq 1$,
\begin{equation}\nonumber
\mathbb{E}\big( \big\langle \xi^N_t-\xi^N_s,e_k \big\rangle^4 \big) \leq C\vert k\vert^8\vert t-s\vert^2\quad \text{for all } k\in\mathbb{Z}^2_0;
\end{equation}
again, we only need to show the estimate for $\vert k\vert\leq N$ and by Lemma \ref{fundamental lemma sec5} it holds
\begin{equation}\nonumber\begin{split}
\big\langle \xi^N_t-\xi^N_s,e_k \big\rangle
= & \int_s^t \big\langle \big(K\ast\xi_r^N \big)\cdot\nabla e_k,\xi_r^N \big\rangle\,\d r + \frac{\varepsilon_N^2}{2}\int_s^t \big\langle C_N e_k,\xi^N_r \big\rangle\,\d r\\
& -\varepsilon_N\sum_{k\in\mathbb{Z}^2_0} \theta_k^N\int_s^t \big\langle\sigma_k\cdot\nabla e_k,\xi_r^N \big\rangle\,\d W^k_r.
\end{split}\end{equation}
The first and the last term on the right hand side can be estimated similarly to Lemma \ref{lem-estimate} using respectively the H\"older and Burkholder inequality. For the term involving the corrector $C_N$, thanks to the energy identity \eqref{energy conservation sec5} and estimate \eqref{estimate approx. corrector sec5}, we have
\begin{equation}\nonumber\begin{split}
\Big\vert \frac{\varepsilon_N^2}{2}\int_s^t \big\langle C_Ne_k,\xi_r^N \big\rangle\,\d r\Big\vert
& \leq \frac{\varepsilon_N^2}{2} \Vert C_Ne_k\Vert_{L^2}\int_s^t \big\Vert \xi_r^N \big\Vert_{L^2}\,\d r\\
& \leq 2\nu \Vert \nabla^2 e_k\Vert_{L^2} \vert t-s\vert\, \Vert \xi_0\Vert_{L^2}
\leq C\vert k\vert^2 \vert t-s\vert,
\end{split}\end{equation}
which implies the conclusion.
\end{proof}

We are now ready to complete the

\begin{proof}[Proof of Theorem \ref{thm sec5}]
We only sketch the proof, highlighting the passages which require to be handled differently from the previous sections. Observe first of all that $\big\{\xi^N \big\}_{N\geq 1}$ is a sequence of variables all defined on the same probability space, therefore convergence in probability to a deterministic limit is equivalent to convergence in law to it. As the sequence $\big\{Q^N \big\}_{N\geq 1}$ is tight, it suffices to show that any weakly convergent subsequence we extract converges to $\delta_{\xi_\cdot}$, $\xi$ being the
unique solution of \eqref{NSE}. Assume we have extracted a (not relabelled) subsequence $\xi^N$ whose laws $Q^N$ are converging in the topology of $C\big([0,T];H^-(\mathbb{T}^2) \big)\cap L^2_{R,w}$ to the law  $Q$ of a random variable $\tilde\xi$. Then $\Pi_N \xi_0 \to \xi_0$ in $L^2$ and the convergence of the nonlinear term and the stochastic integral can be treated in the same way as in Section \ref{sec4}. Finally, given a countable dense set $\{\phi_n\}_n$, Proposition \ref{fundamental lemma sec5} implies that, for all $n$,
\begin{equation}\nonumber
\frac{\varepsilon_N^2}{2}\int_0^\cdot \big\langle C_N \phi_n,\xi^N_s \big\rangle\,\d s \to \nu \int_0^\cdot \big\langle \Delta \phi_n, \xi^N_s \big\rangle\,\d s\quad \text{in law}.
\end{equation}
Thus the proof is complete.
\end{proof}

\begin{remark}\label{rem end sec5}
In this case the tightness of $\big\{Q^N \big\}_N$ in $C([0,T];H^-)$ is optimal, in the sense that it is not possible to prove tightness in $C\big([0,T]; L^2\big)$. Indeed, if this were true, since the sequence $\xi^N$ satisfies \eqref{energy conservation sec5}, the same should hold for the limit $\xi$, namely $\Vert\xi_t\Vert_{L^2}$ being constant; but we know that $\xi$ is a solution of Navier--Stokes equation, which is dissipative.
\end{remark}

\section{Consequences of the scaling limit} \label{sec-consequences}

In this section we discuss some implications of our scaling limit on the stochastic 2D Euler equations \eqref{SEE-approx-formulation}, including the approximate weak uniqueness, the existence of recovery sequences for Euler equations and a ``weak quenched mixing property'' of the weak solutions. We also give a discussion on possible dissipation of enstrophy in Section \ref{subsec-anomalous}.

\subsection{Approximate uniqueness}\label{subsec-uniqueness}

Uniqueness of solutions for 2D Euler equations when vorticity is in $L^{2}$ is
a famous open problem. In view of certain regularization by noise results,
where uniqueness is restored by a suitable noise, it is natural to ask whether
a suitable noise may provide uniqueness, at least in law, for the solution of
the corresponding stochastic 2D Euler equations with vorticity in $L^{2}$.
We cannot prove such a strong result but we identify a new kind of property
which we may call ``approximate uniqueness'' in law. The precise statement is
given in Corollary \ref{corollary approx uniq} below; roughly speaking it
claims that all different solutions of a suitable stochastic 2D Euler
equations, with a given initial vorticity in $L^{2}$, are
very close to each other in law; for any degree of closedness we find a noise
with such property.

On the family of all Borel probability measures on $C\left(  \left[
0,T\right]  ;H^{-}\right)  $, let $d\left(  \cdot,\cdot\right)  $ be a
distance that metrizes weak convergence.

We fix $\xi_0\in L^2$. For every $N$, let $\mathcal{C}_{N}$ be the class of all weak solutions of
equation \eqref{SEE-approx-formulation} with the initial condition $\xi_0\in L^2$ and satisfying \eqref{energy estimate thm-existence}; moreover, let $\mathcal{C}= \bigcup_{N\in\mathbb{N}} \, \mathcal{C}_{N}$. We denote by $Q_{N}$ the elements of $\mathcal{C}_{N}$ and
generically by $Q$ those of $\mathcal{C}$, interpreting weak solutions as
measures on the path space $C\left(  \left[  0,T\right]  ;H^{-}\right)  $.

\begin{definition}\label{definition approx uniq}
The family of weak solutions $\left\{  Q;Q\in\mathcal{C}\right\}  $ is said to converge
to a probability measure $\mu$ on $C\left(  \left[  0,T\right]  ;H^{-}\right)
$ if, for every $\epsilon>0$, there is $N_{0}\in\mathbb{N}$ such that for all
$N\geq N_{0}$, it holds $d\left(  Q_{N},\mu\right)  <\epsilon$ for all $Q_{N}%
\in\mathcal{C}_{N}$.
\end{definition}

\begin{theorem}\label{thm approx uniq}
Given $\xi_{0}\in L^{2}$, the family of weak solutions $\left\{
Q;Q\in\mathcal{C}\right\}  $ converges to $\delta_{\xi}$ on $C\left(  \left[
0,T\right]  ;H^{-}\right)  $, where $\xi$ is the unique solution of the
deterministic Navier--Stokes equations \eqref{NSE}.
\end{theorem}

\begin{proof}
We argue by contradiction. Assume there is $\epsilon>0$ such that for every
$k\in\mathbb{N}$ there exist $N_{k}\geq k$ and $Q_{N_{k}}\in\mathcal{C}%
_{N_{k}}$ with the property $d\left(  Q_{N_{k}},\delta_{\xi}\right)
\geq\epsilon$. The family $\left\{  Q_{N_{k}}\right\}  _{k\in\mathbb{N}}$ is
tight on $C([0,T], H^-)$ (for reasons similar to those proved above for a generic sequence of the
form $\{  Q_{n} \}  _{n\in\mathbb{N}}$). Hence it has a subsequence
$\big\{  Q_{N_{k_{l}}}\big\}  _{l\in\mathbb{N}}$, where we may assume
$\left\{  N_{k_{l}}\right\}  $ increasing, which converges weakly, thus to
$\delta_{\xi}$ by the argument developed above. \ This is in contradiction with
$d\left(  Q_{N_{k}},\delta_{\xi}\right)  \geq\epsilon$ for every
$k\in\mathbb{N}$.
\end{proof}

\begin{corollary}
\label{corollary approx uniq} For every $\epsilon>0$, there is $N_{0}%
\in\mathbb{N}$ such that for all $N\geq N_{0}$, we have $d\left(  Q_{N}%
,Q_{N}^{\prime}\right)  <\epsilon$ for all $Q_{N},Q_{N}^{\prime}\in
\mathcal{C}_{N}$.
\end{corollary}

\begin{proof}
It follows from the previous theorem by triangle inequality.
\end{proof}

\begin{remark}\label{remark wassertein dist}
If we denote by $d_p$ the $p$-th Wasserstein distance for Borel probability measures on $C([0,T];H^-)$, then Theorem \ref{thm-main} implies convergence of $Q_N$ to $\delta_{\xi}$ in the $p$-th Wasserstein distance, for any $p<\infty$. To see this, we can consider by Skorokhod Theorem a sequence $\tilde{\xi}^N$ distributed as $Q_N$, converging $\tilde{\mathbb{P}}$-a.s. to $\xi$ and satisfying the energy bound \eqref{energy estimate thm-existence}; by dominated convergence this implies
$$\lim_{N\to\infty} \tilde{\mathbb{E}}\Big[ \big\Vert \tilde\xi^N-\xi \big\Vert_{C([0,T];H^{-\delta})}^p\Big]\to 0$$
for any $\delta>0$ and $p<\infty$. In particular it is easy to see that Definition \ref{definition approx uniq}, Theorem \ref{thm approx uniq} and Corollary \ref{corollary approx uniq} still hold if we replaced $d$ by $d_p$.
\end{remark}

\subsection{Recovery sequences for Euler equations}\label{sec 6.2}

We are now going to show that, given any viscosity solution $\xi$ of Euler equations, we can find a suitable sequence $\xi^N$ of solutions of \eqref{SEE-approx-formulation} such that their laws $Q_N$ converge to $\delta_\xi$. This may be seen as a result of existence of recovery sequences, in a nice parallelism with the theory of $\Gamma$-convergence; we stress however that no variational problems are involved in our setting and this is merely an analogy. This result may help understanding the structure of viscosity solutions of Euler equations, deducing their properties from those of the sequence $\big\{ \xi^N \big\}_{N\in \N}$.

We consider a fixed sequence $\theta^N\in \ell^2$ satisfying the usual conditions and a fixed initial data $\xi_0\in L^2$. However we now allow the parameter $\nu$ to vary on $(0,+\infty)$; for fixed $\nu$, $\varepsilon_N$ depends on $\nu$ and $\theta^N$ in the usual way. We denote by $\xi^\nu$ the unique solution of Navier--Stokes with initial data $\xi_0$ and coefficient $\nu$; as in the previous section, we identify any solution of \eqref{SEE-approx-formulation} satisfying \eqref{energy estimate thm-existence} with a Borel probability measure on $C([0,T];H^-)$ and we denote by $d(\cdot,\cdot)$ the distance which metrizes weak convergence. We denote by $\mathcal{C}_{N,\nu}$ the set of laws of weak solutions of \eqref{SEE-approx-formulation} satisfying \eqref{energy estimate thm-existence}, with initial data $\xi_0$ and with respect to the parameters $\theta^N$, $\nu$; a generic element of $\mathcal{C}_{N,\nu}$ is denoted by $Q_{N,\nu}$.

We define $\mathcal{H}$ to be the set of viscosity solutions of Euler equations with initial data $\xi_0$, namely $\xi\in \mathcal{H}$ if there exists a sequence $\nu_n\to 0$ such that $\xi^{\nu_n}\to \xi$ in $C([0,T];H^-)$; if uniqueness of viscosity solutions of Euler were true, than $\mathcal{H}$ would consist of a singleton.

\begin{corollary}\label{cor sec 6.2}
For any $\xi\in \mathcal{H}$ there exist sequences $\nu_i\downarrow 0$, $N_i\uparrow\infty$ such that
$$ \lim_{i\to\infty} d(Q_{N_i,\nu_i},\delta_\xi) = 0.$$
\end{corollary}

\begin{proof} Since $\xi\in\mathcal{H}$, there exists a sequence $\nu_i\downarrow 0$ such that $\xi^{\nu_i}\to \xi$. By Theorem \ref{thm approx uniq}, for fixed $\nu_i$, we can find $N_i$ and an element $Q_{N_i,\nu_i}$ such that $d(Q_{N_i,\nu_i} , \delta_{\xi^{\nu_i}})\leq 1/i$; moreover, since we can construct the sequence inductively, we can always take $N_{i+1}\geq N_i$. Then by the triangle inequality,
$$ d(Q_{N_i,\nu_i},\delta_\xi)\leq d(Q_{N_i,\nu_i},\delta_{\xi^{\nu_i}}) + d(\delta_{\xi^{\nu_i}},\delta_{\xi}) \leq \frac{1}{i} + \Vert \xi^{\nu_i} -\xi\Vert_{C([0,T];H^-)}$$
and the conclusion follows.
\end{proof}

\begin{remark}
Similarly to Remark \ref{remark wassertein dist}, the result still holds if we work with the $p$-th Wasserstein distance $d_p$ instead of $d$, for any $p<\infty$.
\end{remark}

Next, we consider two sequences $\nu_i\to 0$ and $N_i\to\infty$, and for any $i$ an element $Q_{N_i,\nu_i}\in \mathcal{C}_{N_i,\nu_i}$. Using the same arguments in the previous sections, tightness of $\{Q_{N_i,\nu_i}\}_i$ in $C([0,T];H^-)$ can be shown; by Prohorov theorem we can therefore extract a subsequence which is weakly convergent to some probability law $Q$. Then, repeating the arguments in Section \ref{sec4} and observing that this time also the corrector $\nu_i\Delta$ is infinitesimal, we find that almost every realization of $Q$ is a weak solution of deterministic Euler equations with initial data $\xi_0$. Since uniqueness in this case is not known, we cannot conclude that $Q$ is of the form $\delta_\xi$; rather it is a probability distribution on the weak solutions of Euler equation starting at $\xi_0$ -- a \textit{superposition solution}. Observe that in the above argument in principle we did not need to vary $N$: convergence of a subsequence to a superposition solution of deterministic Euler equations also holds if we considered a sequence $Q_{N,\nu_i}\in\mathcal{C}_{N,\nu_i}$ with $N$ fixed. However, the scaling limits we have obtained suggest that varying $N$ should allow to deduce non trivial properties in the limit which are not necessarily present for $N$ fixed; in particular, Corollary \ref{cor sec 6.2} leads us to the following conjecture.

\begin{conjecture}
For any weakly convergent sequence $\{Q_{N_i,\nu_i}
\}_i$, the limit $Q$ is a probability measure supported on $\mathcal{H}$, the set of viscosity solutions of Euler equations starting at $\xi_0$.
\end{conjecture}

\subsection{Weakly quenched exponential mixing properties}\label{subsec-mixing}

The multiplicative transport noise in Stratonovich form used above to perturb 2D
Euler equations is formally vorticity-conservative but not formally
energy-conservative. In general, the energy budget is not clear, namely we
cannot say whether such noise increases or decreases the energy. Due to our
convergence result to the Navier--Stokes equations, however, we can state an
energy-dissipation result, in the precise form of Corollary
\ref{Corollary energy dissipation} below.

\begin{remark}
To avoid misunderstandings, we are not claiming that this noise produces an
anomalous dissipation. Such property means a true dissipation when the
equation is formally energy-conservative. Our noise is not formally
energy-conservative. Thus the only relevant information of Corollary
\ref{Corollary energy dissipation} below is to clarify in which direction
energy goes.
\end{remark}

On the torus $\T^2$, for the unique solution $\xi_{t}\in C\left( [0,T];L^{2}\right)  $
of the deterministic Navier--Stokes equations \eqref{NSE} with
initial condition $\xi_{0}\in L^{2}$, we have
\begin{align*}
\frac{\d}{\d t} \left\Vert \xi_{t}\right\Vert _{L^{2}}^{2}+\alpha\left\Vert
\xi_{t}\right\Vert _{L^{2}}^{2}  & \leq0 , \\
\frac{\d}{\d t} \left\Vert u_{t}\right\Vert _{L^{2}}^{2}+\alpha\left\Vert
u_{t}\right\Vert _{L^{2}}^{2}  & \leq0
\end{align*}
where $\alpha= 8\nu \pi^2$, as a consequence of the inequality $\left\langle -\Delta
f,f\right\rangle \geq 4\pi^2 \left\Vert f\right\Vert _{L^{2}}^{2}$ for smooth $f$. It follows that
\begin{align*}
\left\Vert \xi_{t}\right\Vert _{L^{2}}^{2}  & \leq e^{-\alpha t}\left\Vert
\xi_{0}\right\Vert _{L^{2}}^{2}, \\
\left\Vert u_{t}\right\Vert _{L^{2}}^{2}  & \leq e^{-\alpha t}\left\Vert
u_{0}\right\Vert _{L^{2}}^{2}.
\end{align*}

\begin{definition}
For every $\xi_{\cdot}\in C\left(  \left[  0,T\right]  ;H^- \right)  $, we
call \textit{energy profile} the real valued continuous function
\[
e(t)  :=\frac{1}{2}\left\Vert K\ast\xi_{t}\right\Vert _{L^{2}%
}^{2}, \quad t\in [0,T]  .
\]

\end{definition}

The map $\xi_{\cdot}\mapsto e\left(  \cdot\right)  $ from $C\left(  \left[
0,T\right]  ;H^-\right)  $ to $C\left(  \left[  0,T\right]  ;\mathbb{R}%
\right)  $ is continuous. Given $\xi_{0}\in L^{2}$, the energy profile of
the unique solution $\xi$ of the deterministic Navier--Stokes equations \eqref{NSE}
satisfies $e\left(  t\right)  \leq e^{-\alpha t}e\left(  0\right)  $.
Concerning solutions $\xi^{N}$ of the stochastic 2D Euler equations, always with initial
condition $\xi_{0}\in L^{2}$, since their trajectories are of class $C\left(
\left[  0,T\right]  ;H^{-}\right)  $, the energy profile $e_{N}(t) = \frac12 \big\| K\ast\xi^N_{t}\big\|_{L^2}\, (t\in [0,T])$ is well defined also for them, being in this case a real-valued continuous stochastic process. Recall that $\xi^{N}$ converge in law to $\xi$ on $C([0,T]; H^{-})$. Thanks to the stability of convergence in law by composition with continuous functions, we deduce that $e_N$ converge in law to $e$ on $C([0,T] ; \mathbb{R})  $.

\begin{corollary}\label{Corollary energy dissipation}
For every $\epsilon>0$,
\[
\lim_{N\rightarrow\infty}\mathbb{P}\left(  e_{N}(t)  \leq
e^{-\alpha t} (e(0)  +\epsilon) \mbox{ for all } t\in [0,T] \right)  =1.
\]

\end{corollary}

\begin{proof}
Given $\epsilon>0$, by the above discussions,
\[
\lim_{N\rightarrow\infty} \mathbb{P}\left(  \|e_{N}(\cdot) - e(\cdot)\|_{C([0,T] ; \mathbb{R})} \leq
\epsilon\right)  =1.
\]
Since $e(t) \leq e^{-\alpha t} e(0)$ for all $t\in [0,T]$, we have
  \begin{equation}\label{Corollary energy dissipation-1}
  \lim_{N\rightarrow\infty} \mathbb{P}\big( e_{N}(t) \leq  e^{-\alpha t} e(0) + \epsilon \mbox{ for all } t\in [0,T] \big)  =1.
  \end{equation}
Note that $e^{-\alpha t} e(0) + e^{-\alpha T}\epsilon \leq e^{-\alpha t} (e(0) +\epsilon)$ for all $t\in [0,T]$, replacing $\epsilon$ by $e^{-\alpha T}\epsilon$ in \eqref{Corollary energy dissipation-1} gives us the result.
\end{proof}

We cannot state a similar result for the enstrophy profile%
\[
i( t)  := \Vert \xi_{t}\Vert_{L^{2}}^{2},
\]
even if it is well defined for both $\xi^{N}$ and $\xi$. Indeed, $\xi\in C\left(
\left[  0,T\right]  ;L^{2}\right)  $, hence $i\left(  \cdot\right)  \in
C\left(  \left[  0,T\right]  ;\mathbb{R}\right)  $, but we only know that
$\xi^{N}\in L^{\infty}\left(  \left[  0,T\right]  ;L^{2}\right)  $ and that
$\xi^{N}$ converges in law to $\xi$ in the strong topology of $C\left(
\left[  0,T\right]  ;H^{-}\right)  $. Thus we cannot say that enstrophy is
dissipated (in a probabilistic sense). If true, this would be a result of
anomalous enstrophy dissipation, because formally the enstrophy is conserved
by the stochastic dynamics.

However, for the solution to the deterministic 2D Navier--Stokes equation, we have%
\begin{align*}
\left\Vert \xi_{t}\right\Vert _{H^{-\delta}}^{2}  & \leq \left\Vert \xi
_{t}\right\Vert _{L^2}^{2} \leq e^{-\alpha t}\left\Vert \xi_{0}\right\Vert _{L^2%
}^{2}%
\end{align*}
and the convergence in law of $\xi^N$ to $\xi$ in $C\left([0,T]  ;H^{-}\right)
$. Repeating the argument above gives us the asymptotically exponential decay of vorticity in negative Sobolev norms.

\begin{proposition}\label{prop-mixing-in-probab}
For every $\epsilon,\delta>0$,
\[
\lim_{N\rightarrow\infty}\mathbb{P}\left(  \left\Vert \xi_{t}^{N}\right\Vert
_{H^{-\delta}}^{2}\leq e^{-\alpha t} \big(\Vert \xi_{0}\Vert _{L^{2}}^2 +\epsilon \big)  \mbox{ for all } t\in [0,T] \right)  =1.
\]
\end{proposition}

In the rest of this subsection, to avoid technical problems (cf. Remark \ref{rem-theta}), we take $\theta^N_k = {\bf 1}_{\{|k|\leq N\}},\, k\in \Z^2_0$. Denote by $L^\infty_0 = L^\infty_0(\T^2)$ the space of functions in $L^\infty (\T^2)$ with zero mean. Then for any $\xi_0 \in L^\infty_0$, by \cite[Theorem 2.10]{BrFlMa}, the following stochastic Euler equation on $\T^2$
  $$\d \xi^N + u^N\cdot\nabla\xi^N\,\d t= \eps_N \sum_{|k|\leq N} \theta^N_k \sigma_{k} \cdot \nabla\xi^N \circ \d W^{k}, \quad \xi^N|_{t=0} = \xi_0$$
admits a unique solution $\xi^{N,\xi_0}$ in $L^\infty_0$; moreover, \cite[Theorem 2.14]{BrFlMa} implies that the equation of characteristics
  $$\d X_t= u_t^N(X_t)\,\d t - \eps_N \sum_{|k|\leq N} \theta^N_k \sigma_{k}(X_t)\circ \d W^k_t$$
generates a stochastic flow $\varphi^{N, \xi_0}_t$ of homeomorphisms on $\T^2$, such that, $\P$-a.s. for all $t\geq 0$ and $x\in \T^2$, it holds
  \begin{equation}\label{repres-flow}
  \xi^{N,\xi_0}_t(\omega, x) = \xi_0\big(\varphi^{N, \xi_0}_{-t} (\omega, x)\big),
  \end{equation}
where $\varphi^{N, \xi_0}_{-t} (\omega, \cdot)$ is the inverse map of $\varphi^{N, \xi_0}_{t} (\omega, \cdot)$. Moreover, $\P$-a.s., the Lebesgue measure on $\T^2$ is invariant under the stochastic flow $\varphi^{N, \xi_0}_t$ for all $t\geq 0$. The above formula implies that, $\P$-a.s., the norms $\big\| \xi^{N,\xi_0}_t \big\|_{L^p}\ (p>1)$ are preserved. In particular, taking $p=2$ and $\delta>0$, by the interpolation inequality, $\P$-a.s.,
  $$\|\xi_0\|_{L^2}^2 = \big\| \xi^{N,\xi_0}_t \big\|_{L^2}^2 \leq \big\| \xi^{N,\xi_0}_t \big\|_{H^\delta} \big\| \xi^{N,\xi_0}_t \big\|_{H^{-\delta}}\quad \mbox{for all } t>0.$$
Combining this with Proposition \ref{prop-mixing-in-probab}, we obtain the asymptotically exponential increase of vorticity in positive Sobolev norms.

\begin{corollary}
Given $\xi_0\in L^\infty_0$, for any $\delta>0$ and $T>0$,
  $$\lim_{N\to \infty} \P \Big( \big\| \xi^{N,\xi_0}_t \big\|_{H^\delta} \geq \frac12 e^{\alpha t/2} \|\xi_0\|_{L^2} \mbox{ for all } t\in [0,T] \Big) =1. $$
\end{corollary}

Next we will deduce a result on the weakly quenched mixing behavior of the stochastic flows $\varphi^{N, \xi_0}_t$.

\begin{lemma}
Let $\xi _{0}\in L^\infty_0$. There exists a null set $\mathcal N\subset \Omega $ such that for all $\omega \in \mathcal N^{c}$, for all $N\in \N$, for every $ f\in H^\delta$ and all $t\geq 0$, we have
\begin{eqnarray*}
\left\vert \int_{\mathbb{T}^{2}}f\left( \varphi _{t}^{N, \xi _{0}}( \omega
,x) \right) \xi _{0}( x) \,\d x
\right\vert  &\leq &\big\Vert \xi _{t}^{N, \xi _{0}}( \omega) \big\Vert
_{H^{-\delta}} \Vert f \Vert_{H^{\delta}}.
\end{eqnarray*}
\end{lemma}

\begin{proof}
For any $N\in \N$, there exists a null set $\mathcal N_{N} \subset \Omega$ such that for all $\omega \in \mathcal N_{N}^{c}$, for all $t\geq 0$, the formula \eqref{repres-flow} holds and the Lebesgue measure is invariant under the map $\varphi _{t}^{N, \xi _{0}}(\omega, \cdot)$. For every $f\in H^\delta$, we have
  $$\aligned
  \left\vert \int_{\mathbb{T}^{2}}f\left( \varphi _{t}^{N, \xi _{0}}( \omega,x) \right) \xi _{0}( x) \,\d x \right\vert
  & =\left\vert \int_{\mathbb{T}^{2}}\xi _{t}^{N, \xi _{0}}( \omega ,x) f( x) \,\d x \right\vert \\
  &\leq \big\Vert \xi _{t}^{N, \xi _{0}}( \omega) \big\Vert_{H^{-\delta}} \Vert f\Vert_{H^{\delta}}.
  \endaligned$$
Now it is clear that the assertion holds.
\end{proof}

The above result plus Proposition \ref{prop-mixing-in-probab} gives us the weakly quenched exponential mixing property of the stochastic flows $\varphi^{N, \xi_0}_t$.

\begin{corollary}
Under the previous notations, for every $\xi _{0}\in L_0^\infty, f\in H^\delta$, for every $\epsilon >0$,
\[
\lim_{N\rightarrow \infty } \P\left( \left\vert \int_{\mathbb{T}^{2}}f\left(
\varphi _{t}^{N, \xi _{0}} ( x) \right) \xi _{0}( x)
 \,\d x \right\vert \leq e^{-\alpha t/2} (\Vert \xi_{0} \Vert_{L^2}+\epsilon) \|f \|_{H^\delta} \mbox{ for all } t\in [0,T]
\right) =1.
\]
\end{corollary}

\subsection{Further discussions on anomalous dissipation of enstrophy}\label{subsec-anomalous}

As already pointed out earlier, the fact that the vorticity processes $\xi^N$ converge to a limit $\xi$ which is explicitly dissipating suggests that a partial dissipation should already take place at the level of $\xi^N$; the problem is that we only have convergence in $C([0,T];H^-)$ and not in $C([0,T];L^2)$, which does not allow to conclude.

The problem is not only technical: examples of processes $\xi^N$ which preserve vorticity almost surely but are converging in $C([0,T];H^-)$ to the solution $\xi$ of deterministic Navier--Stokes equations can be indeed found. One example is given by the processes from Section \ref{sec5}, as pointed out in Remark \ref{rem end sec5}.

Another example is the following: let $\xi_0\in C^\infty(\mathbb{T}^2)$ and the sequence $\theta^N$ be taken as in the last subsection, that is, for each fixed $N$, only a finite number of $\theta^N_k$ are non zero. Then the solution $\xi^N$ will preserve spatial regularity over time, for instance because $\Vert \xi^N\Vert_{L^\infty}$ can be controlled uniformly and Beale--Kato--Majda criterion can be applied, see \cite{CrFlHo}. This implies that the formal computation on vorticity invariance is actually rigorous and so $\big\Vert \xi^N_t \big\Vert_{L^2}=\Vert\xi_0\Vert_{L^2}$ for all $t>0$.

The above examples show that our scaling limit does not a priori give any information on whether anomalous dissipation will take place. It definitely does not take place for all solutions, but it might at least for \textit{some} of them. Before proceeding further, let us give a rigorous definition.

\begin{definition}
Let $\xi_\cdot$ be a weak solution of \eqref{SEE-formulation} satisfying \eqref{energy estimate thm-existence}. We say that anomalous dissipation of enstrophy takes place with positive probability if, for some $t\in [0,T]$, it holds
\begin{equation}\label{condition def anomalous dissipation}
\mathbb{P}\big(\Vert \xi_t\Vert_{L^2}<\Vert \xi_0\Vert_{L^2} \big)>0.
\end{equation}
\end{definition}

\begin{remark}
Since $\Vert\xi_t\Vert\leq \Vert \xi_0\Vert_{L^2}$ with probability one, condition \eqref{condition def anomalous dissipation} is equivalent to requiring that, for some $t\in [0,T]$, $\mathbb{E}(\Vert \xi_t\Vert_{L^2}) <\mathbb{E}(\Vert \xi_0\Vert_{L^2})$. Such a quantity might be easier to handle because although we do not know whether $\xi$ has trajectories in $C([0,T];L^2)$, the map $t\mapsto \mathbb{E}(\Vert \xi_t\Vert_{L^2})$ might be continuous as an effect of the averaging. However, condition \eqref{condition def anomalous dissipation} is not equivalent to
\begin{equation}\nonumber
\mathbb{P}\big( \Vert \xi_t\Vert_{L^2}<\Vert \xi_0\Vert_{L^2} \text{ for some } t\in [0,T]\big)>0;
\end{equation}
while the latter seems a more natural definition of anomalous dissipation, the fact that it involves evaluation on an uncountable set $[0,T]$ for a process $\xi_t$ with possibly not continuous trajectories in $L^2$ (not even right/left continuous) makes it very difficult to be handled.
\end{remark}

The occurrence of anomalous dissipation might rely on the kind of noise we use. Here we restrict to the case of a noise constructed from $\theta\in\ell^2$ and $\{\sigma_k\}_{k\in \Z_0^2}$ as before, but observe that this is a very specific choice: it is an isotropic, divergence-free noise whose covariance operator is a Fourier multiplier; this leaves open the question whether other choices of noise might be better suited for obtaining an anomalous dissipation effect. One can also consider similar problems for equations on a 2D domain; for the moment we do not have any idea on them: indeed, the search of a family of divergence free vector fields on a domain with the property \eqref{useful-equality} is not an easy task. In any case it would be interesting to give an answer to the following:

\begin{problem}\label{problem1}
Do there exist an initial data $\xi_0\in L^2$, a family of coefficients $\theta\in\ell^2$ and an associated solution $\xi$ which displays anomalous dissipation of enstrophy?
\end{problem}

A different question, in the case of a positive answer for Problem \ref{problem1}, is related to anomalous dissipation occurring for \textit{all} initial data.

\begin{problem}\label{problem4}
Does there exist a family of coefficients $\theta$ such that any solution of \eqref{SEE-formulation} satisfying \eqref{energy estimate thm-existence}, for any initial data $\xi_0\in L^2$, displays anomalous dissipation of enstrophy with positive probability?
\end{problem}

Clearly, if a positive answer to Problem \ref{problem4} could be given, then the previous examples would show that $\theta$ cannot consist of all but a finite number of $\theta_k$ being $0$; more refined arguments show that in general $\theta_k$ cannot decay too fast as $k\to\infty$. On the other hand, condition $\theta\in\ell^2$, which is required for the equation to be meaningful, implies that such decay cannot be too slow either. It would be interesting to explore the case of $\theta_k$ decaying ``almost as slowly as possible'', for instance taking
$$ \theta_k \sim \frac{1}{\vert k\vert \log\vert k\vert}.$$
Observe however that dealing with such a choice of $\theta$ is highly non trivial: uniqueness of solutions of \eqref{SEE-formulation} for such $\theta$, even in the case of smooth initial data, is not known.

\bigskip

\noindent \textbf{Acknowledgements.} The third named author is grateful to the financial supports of the grant ``Stochastic models with spatial structure'' from the Scuola Normale Superiore di Pisa, the National Natural Science Foundation of China (Nos. 11571347, 11688101), and the Youth Innovation Promotion Association, CAS (2017003).


\begin{thebibliography}{9}
\bibitem{ACM14} G. Alberti, G. Crippa, A. L. Mazzucato, Exponential self-similar mixing and loss of regularity for continuity equations. \emph{C. R. Math. Acad. Sci. Paris} \textbf{352} (2014), no. 11, 901--906.

\bibitem{ACM19} G. Alberti, G. Crippa, A. L. Mazzucato, Exponential self-similar mixing by incompressible flows. \emph{J. Amer. Math. Soc.} \textbf{32} (2019), no. 2, 445--490.

\bibitem{AC} S. Albeverio, A. B. Cruzeiro, Global flows with invariant (Gibbs) measures for Euler and Navier--Stokes two-dimensional fluids. \textit{Comm. Math. Phys.} \textbf{129} (1990), 431--444.

\bibitem{AF} S. Albeverio, B. Ferrario, Uniqueness of solutions of the stochastic Navier--Stokes equation with invariant measure given by the enstrophy. \emph{Ann. Probab.} \textbf{32} (2004), 1632--1649.

\bibitem {AttFla}S. Attanasio, F. Flandoli, Zero-noise solutions of linear transport equations without uniqueness: an example, \textit{C. R. Math. Acad. Sci. Paris} \textbf{347} (2009), no. 13--14, 753--756.

%\bibitem{Beck} W. Beckner, Inequalities in Fourier analysis, \emph{Ann. Math.} \textbf{102} (1975), no. 1, 159--182.

\bibitem{Billingsley} P. Billingsley, Convergence of Probability Measures. Second edition. Wiley Series in Probability and Statistics: Probability and Statistics. A Wiley-Interscience Publication. \emph{John Wiley \& Sons, Inc., New York,} 1999.

\bibitem{Bre} H. Brezis, Functional analysis, Sobolev spaces and partial differential equations. \emph{Springer Science and Business Media} (2010).

\bibitem{BrFlMa} Z. Brze\'{z}niak, F. Flandoli, M. Maurelli, Existence and uniqueness for stochastic 2D Euler flows with bounded vorticity. \emph{Arch. Ration. Mech. Anal.} \textbf{221} (2016), no. 1, 107--142.

\bibitem{CoKiRyZl} P. Constantin, A. Kiselev, L. Ryzhik, A. Zlatos. Diffusion and mixing in fluid flow. \emph{Ann. of Math.} \textbf{168} (2008), no. 2, 643--674.

\bibitem{CrFlHo} D. Crisan, F. Flandoli, D.D. Holm, Solution properties of a 3D stochastic Euler fluid equation. \emph{ J. Nonlinear Sci.} \textbf{29} (2019), no. 3, 813--870.

\bibitem{DaPD} G. Da Prato, A. Debussche, Two-Dimensional Navier--Stokes Equations Driven by a Space--Time White Noise. \emph{J. Funct. Anal.} \textbf{196} (2002), 180--210.

\bibitem{DaPZ} G. Da Prato, J. Zabczyk,  Stochastic equations in infinite dimensions, \emph{Cambridge university press} (2014).

\bibitem {DeLellSze} C. De Lellis, L. Sz\'{e}kelyhidi, The Euler equations as
a differential inclusion. \textit{Ann. of Math.} (2) \textbf{170} (2009), no.
3, 1417--1436.

\bibitem{DiPLio} R.J. DiPerna, P.L. Lions, Ordinary differential equations, transport theory and Sobolev spaces. \emph{Invent. Math.} \textbf{98} (1989), no. 3, 511--547.

%\bibitem{DrHo} T.D. Drivas, D.D. Holm, Circulation and energy preserving stochastic fluids. arXiv:1808.05308.

%\bibitem{FlaLuo-JGM} F. Flandoli, D. Luo, Euler-Lagrangian approach to 3D stochastic Euler equations. \emph{J. Geom. Mech.} \textbf{11} (2019), no. 2, 153--165.

\bibitem{FlaLuo-1} F. Flandoli, D. Luo, $\rho$-white noise solutions to 2D stochastic Euler equations. \emph{Probab. Theory Relat. Fields} \textbf{175} (2019), no. 3--4, 783--832.

\bibitem{FlaLuo-2} F. Flandoli, D. Luo, Convergence of transport noise to Ornstein-Uhlenbeck for 2D Euler equations under the enstrophy measure. \emph{Ann. Probab.} (2019), accepted.

%\bibitem{FlaMah} F. Flandoli, A. Mahalov: Stochastic three-dimensional rotating Navier¨C-Stokes equations: averaging, convergence and regularity. \emph{Archive for Rational Mechanics and Analysis} \textbf{205} (2012), no. 1, 195--237.

\bibitem{Gal} L. Galeati, On the convergence of stochastic transport equations to a deterministic parabolic one. arXiv:1902.06960.

\bibitem{GyMa} I. Gy\"ongy, T. Martinez, On stochastic differential equations with locally unbounded drift. \emph{Czechoslovak Mathematical Journal} \textbf{51} (2001) 763--783.

\bibitem{IKX} G. Iyer, A. Kiselev, X. Xu, Lower bounds on the mix norm of passive scalars advected by incompressible enstrophy-constrained flows. \emph{Nonlinearity} \textbf{27} (5) (2014) 973--985.

\bibitem{Krylov} N.V. Krylov, Controlled Diffusion Processes. Translated from the Russian by A. B. Aries.Applications of Mathematics, vol. 14. Springer, New York (1980).

\bibitem{Kur} T.G. Kurtz, The Yamada--Watanabe--Engelbert theorem for general stochastic equations and inequalities. \emph{Electron. J. Probab.} \textbf{12} (2007), 951--965.

\bibitem{Lions} J.L. Lions, Quelques m\'ethodes de r\'esolution des probl\`emes aux limites non lin\'eaires. (French) Dunod; Gauthier--Villars, Paris, 1969.

\bibitem{Luo} D. Luo, Absolute continuity under flows generated by SDE with measurable drift coefficients. \emph{Stochastic Process. Appl.} \textbf{121} (2011), no. 10, 2393--2415.

\bibitem{Seis} C. Seis, Maximal mixing by incompressible fluid flows. \emph{Nonlinearity} \textbf{26} (2013), no. 12, 3279--3289.

\bibitem{Simon} J. Simon, Compact sets in the space $L^p(0,T; B)$. \emph{Ann. Mat. Pura Appl.} \textbf{146} (1987), 65--96.

\bibitem{Teman} R. Temam, Navier--Stokes equations and nonlinear functional analysis. Second edition. CBMS--NSF Regional Conference Series in Applied Mathematics, 66. \emph{Society for Industrial and Applied Mathematics (SIAM), Philadelphia, PA}, 1995.

\bibitem{YZ} Y. Yao and A. Zlato\v{s}, Mixing and un-mixing by incompressible flows. \emph{J. Eur. Math. Soc. (JEMS)} \textbf{19} (2017), no. 7, 1911--1948.

\end{thebibliography}
\end{document}